\documentclass[a4paper,leqno]{amsart}
\usepackage{amssymb,amsthm,amsmath,amsrefs}
\usepackage{graphicx}
\usepackage{ifthen}

\usepackage{lineno}

\usepackage{accents}
\newcommand\thickbar[1]{\accentset{\rule{.4em}{.6pt}}{#1}}

\newcounter{mylisti} \newcounter{mylistii}
\newcounter{nest}
\newcommand{\defaultlabel}{}

\newenvironment{mylist}[1]{%
  \addtocounter{nest}{1}
  \ifthenelse{\value{nest}=1}{%
    \renewcommand{\defaultlabel}{(\roman{mylisti})\hfill}}{%
    \renewcommand{\defaultlabel}{(\alph{mylistii})\hfill}}
  \begin{list}{\defaultlabel}{%
      \ifthenelse{\value{nest}=1}{\usecounter{mylisti}}{%
        \usecounter{mylistii}}
      
      \addtolength{\itemsep}{0.5ex}
      \settowidth{\labelwidth}{#1}
      \setlength{\leftmargin}{\labelwidth}
      \addtolength{\leftmargin}{\labelsep}}}{\addtocounter{nest}{-1}
\end{list}}

\newcommand{\fb}{\ensuremath{\thickbar{f}}}

\newcommand{\bn}{\ensuremath{\mathbb N}}

\newcommand{\br}{\ensuremath{\mathbb R}}

\newcommand{\bz}{\ensuremath{\mathbb Z}}

\usepackage{mathrsfs}

\newcommand{\Ht}{\ensuremath{\widetilde{H}}}

\newcommand{\vu}{\ensuremath{\boldsymbol{\mathrm{u}}}}
\newcommand{\vv}{\ensuremath{\boldsymbol{\mathrm{v}}}}

\newcommand{\abs}[1]{\lvert #1\rvert}
\newcommand{\bigabs}[1]{\big\lvert #1\big\rvert}

\newcommand{\aut}{\operatorname{Aut}}

\newcommand{\diam}{\operatorname{diam}}

\newcommand{\dist}{\ensuremath{\mathrm{dist}}}

\newcommand{\ceil}[1]{\ensuremath{\lceil #1\rceil}}

\newcommand{\meet}{\wedge}

\newcommand{\norm}[1]{\lVert #1\rVert}

\newcommand{\Bignorm}[1]{\Big\lVert #1\Big\rVert}

\newcommand{\symdif}{\bigtriangleup}

\newcommand{\E}{\exists\,}

\newcommand{\vare}{\varepsilon}
\newcommand{\varf}{\varphi}

\renewcommand{\geq}{\geqslant}
\renewcommand{\leq}{\leqslant}

\newcommand{\ds}{\displaystyle}

\newcommand{\Eg}{\textit{E.g.,}\ }
\newcommand{\ie}{\textit{i.e.,}\ }

\newcommand{\cf}{\textit{cf.}\ }

\newcommand{\etc}{\textit{etc.}\ }

\newtheorem{thm}{Theorem}
\newtheorem{mainthm}{Theorem}
\newtheorem*{corD}{Corollary D}

\newtheorem*{mainproblem*}{Problem}

\newtheorem{lem}[thm]{Lemma}
\newtheorem{prop}[thm]{Proposition}

\newtheorem{problem}[thm]{Problem}

\theoremstyle{definition}

\theoremstyle{remark}

\newtheorem*{rem}{Remark}

\newtheorem*{ex}{Example}

\DeclareMathOperator{\cay}{Cay}
\DeclareMathOperator{\clo}{Clo}
\DeclareMathOperator{\cprod}{\square}
\DeclareMathOperator{\la}{La}
\DeclareMathOperator{\tsp}{tsp}

\newcommand{\bin}{\ensuremath{\mathrm{B}}}
\newcommand{\coa}{\ensuremath{\ast}}
\newcommand{\com}{\ensuremath{\mathrm{K}}}
\newcommand{\cyc}{\ensuremath{\mathrm{C}}}
\newcommand{\ham}{\ensuremath{\mathrm{H}}}
\newcommand{\pa}{\ensuremath{\mathrm{P}}}
\newcommand{\rose}{\ensuremath{\mathrm{Ro}}}
\newcommand{\st}{\ensuremath{\mathrm{St}}}

\begin{document}

    
\title{On the bi-Lipschitz geometry of lamplighter graphs}

\author{F.~Baudier}
\address{Department of Mathematics, Texas A\&M University, College
  Station, TX 77843, USA}
\email{florent@tamu.edu}

\author{P.~Motakis}
\address{Department of Mathematics, University of Illinois at
  Urbana-Champaign, Urbana, IL 61801, USA}
\email{pmotakis@illinois.edu}

\author{Th.~Schlumprecht}
\address{Department of Mathematics, Texas A\&M University, College
  Station, TX 77843, USA and Faculty of Electrical Engineering, Czech
  Technical University in Prague,  Zikova 4, 166 27, Prague}
\email{schlump@math.tamu.edu}

\author{A.~Zs\'ak}
\address{Peterhouse, Cambridge, CB2 1RD, UK}
\email{a.zsak@dpmms.cam.ac.uk}

\date{\today}

\thanks{The first named author was supported by the National Science
  Foundation under Grant Number DMS-1800322. The second named author
  was  supported by National Science Foundation under Grant Numbers
  DMS-1600600 and DMS-1912897. The third  named author was supported by the National
  Science Foundation under Grant Numbers DMS-1464713 and
  DMS-1711076. The  fourth author was supported by  the 2018 Workshop
  in  Analysis and Probability at Texas A\&M   University.}


\begin{abstract}
  In this article we start a systematic study of the bi-Lipschitz
  geometry of lamplighter graphs. We prove that lamplighter graphs
  over trees bi-Lipschitzly embed into Hamming cubes with distortion
  at most~$6$. It follows that lamplighter graphs over countable trees
  bi-Lipschitzly embed into $\ell_1$. We study the metric behaviour of
  the operation of taking the lamplighter graph over the
  vertex-coalescence of two graphs. Based on this analysis, we provide
  metric characterizations of superreflexivity in terms of lamplighter
  graphs over star graphs or rose graphs. Finally, we show that the
  presence of a clique in a graph implies the presence of a Hamming cube
  in the lamplighter graph over it. An application is a characterization in terms of a sequence of graphs with uniformly bounded degree of the notion of trivial Bourgain-Milman-Wolfson type for arbitrary metric spaces, similar to Ostrovskii's characterization previously obtained in \cite{ostrovskii:11}.
\end{abstract}

\maketitle

\section{Introduction}

Wreath products of groups provide a wealth of fundamental examples
with various algebraic, spectral and geometric properties. Given two
groups $\Gamma_1$ and $\Gamma_2$, we denote by $\Gamma_2^{(\Gamma_1)}$
the set of all functions $f\colon\Gamma_1\to \Gamma_2$ with finite
support, \ie with $\{x\in\Gamma_1:\,f(x)\neq e_{\Gamma_2}\}$ finite,
where $e_{\Gamma_2}$ is the identity element of $\Gamma_2$. This is a
group with pointwise multiplication. We let
$\lambda\colon \Gamma_1\to\aut\big(\Gamma_2^{(\Gamma_1)}\big)$ denote
the left-regular representation given by $\lambda(x)(f)=f^x$, where
$f^x(y)=f(x^{-1}y)$. The \emph{(restricted) wreath product
  $\Gamma_2\wr \Gamma_1$ of $\Gamma_2$ with $\Gamma_1$ }is then
defined as the semi-direct product
$\Gamma_2^{(\Gamma_1)}\rtimes_\lambda \Gamma_1$. It is the group of
all pairs $(f,x)$, where $f\in\Gamma_2^{(\Gamma_1)}$ and
$x\in\Gamma_1$, equipped with the product
$(f,x)\cdot(g,y)=(fg^x,xy)$. When $\Gamma_2=\bz_2$ (the cyclic group
of order $2$), the wreath product $\bz_2\wr \Gamma_1$ is commonly
referred to as \emph{the lamplighter group of $\Gamma_1$}. We shall
often identify $\bz_2^{(\Gamma_1)}$ with the set of all finite subsets
of $\Gamma_1$. Under this identification, pointwise product becomes
symmetric difference, and hence the group operation of
$\bz_2\wr \Gamma_1$ is given by $(A,x)\cdot(B,y)=(A\symdif xB,xy)$,
where $xB=\{xb:\,b\in B\}$.

The group $\bz_2\wr \bz$ is an example of an amenable group with
exponential growth. Random walks on wreath product groups have been
extensively studied and are well known to exhibit interesting
behaviours. In an influential article~\cite{kv:83},
Ka\u{\i}manovich and Vershik showed that $\bz_2\wr \bz$ is an example
of a group of exponential growth for which the simple random walk on
the Cayley graph has zero speed. The variety of geometric features of
wreath products of groups has also come to play an important role,
sometimes quite unexpectedly, in metric geometry. For instance, the
geometry of $\bz\wr\bz$ is closely connected to the extension of
Lipschitz maps~\cite{naor-peres:11}, and is also used in
distinguishing bi-Lipschitz invariants, namely Enflo type and edge
Markov type~\cite{naor-peres:08}.

In geometric group theory, the theory of compression exponents has
undergone a detailed study, in particular the behaviour of compression
exponents under taking wreath products. Compression exponents were
introduced by Guentner and Kaminker in order to measure how well an
infinite, finitely generated group that does not admit a bi-Lipschitz
embedding into a certain metric space, can be faithfully represented
in it. A deep result of Naor and Peres states that the
$\ell_1$-compression of a lamplighter group over a group with at least
quadratic growth is~$1$. This result includes the case of the planar
lamplighter group $\bz_2\wr \bz^2$. However, it is not known whether
$\bz_2\wr \bz^2$ bi-Lipschitzly embeds into $\ell_1$. This challenging
problem was raised by Naor and Peres
in~\cite{naor-peres:11}. Understanding the $\ell_1$-embeddability of
graphs is motivated by its profound connections with the design of
efficient algorithms for some NP-hard problems
(see~\cite{deza-laurent:97}*{Chapter~10},
\cite{got:18}*{Chapter~8, Chapter~43}, and~\cite{naor:10}). Very
little is known about the bi-Lipschitz embeddability of lamplighter
groups into Banach spaces. The Euclidean distortion of $\bz_2\wr\bz_k$
is of the order $\sqrt{\log k}$. The lower bound was proved
in~\cite{lnp:09} and the upper bound in~\cite{anv:10}. It was shown
in~\cite{naor-peres:08} that $\bz_2\wr\bz_k$ bi-Lipschitzly embeds
into $\ell_1$ with some distortion independent of $k$ (and thus so
does $\bz_2\wr \bz$). In~\cite{RO2018}, it was proved that a Banach space is superreflexive if and only if it does not contain bi-Lipschitz copies of $\bz_2\wr\bz_k$ (for every $k\in \bn$ and with uniformly bounded distortions). In~\cite{csv:12}, Cornulier, Stadler and Valette
proved that for a finitely generated group $\Gamma$ and for a finitely
generated free group $\mathrm{F}$, the equivariant $L_1$-compression
of $\Gamma\wr \mathrm{F}$ is equal to that of $\Gamma$. It follows
from this that $\bz_2\wr \mathrm{F}$ bi-Lipschitzly embeds into
$\ell_1$.

Working with groups might be restrictive because relatively few graphs
can be realized as Cayley graphs of groups. In this paper we consider
the most general graph-theoretic setting and we will be concerned
with the metric geometry of lamplighter graphs. We anticipate that
working in this more flexible framework will be fruitful to construct
new graphs with subtle geometric properties. Moreover, lamplighter
graphs are generalizations of the wreath product construction in group
theory and our results apply to lamplighter groups as well. Indeed, in
the context of graph theory it is possible to define a notion of the
wreath product of two graphs that is compatible with the wreath
product construction in group theory in the sense that the wreath
product of two Cayley graphs of groups is the Cayley graph of the
wreath product of the two groups for a well-chosen set of generators
(\cf~\cite{donno:15}). For practical purposes which will be explained
in the next section, we chose to work with the \emph{walk/switch model
}of the lamplighter graph over a graph $G$, simply denoted
$\la(G)$. Specifically, $\la(G)$ is the graph whose vertex
set consists of all pairs $(A,x)$ where $A$ is a finite subset of
the vertex set of $G$, and $x$ is a vertex of $G$. Vertices $(A,x)$
and $(B,y)$ of $\la(G)$ are joined by an edge if and only if
\emph{either }$A=B$ and $xy$ is an edge in $G$ \emph{or }$x=y$ and
$A\symdif B=\{x\}$. A well known description of this graph is as
follows. Assume there is a lamp attached to each vertex of $G$ and a
lamplighter is able to walk along edges of $G$ and switch lights on
and off. A vertex $(A,x)$ corresponds to the lamplighter standing at
vertex $x$ of $G$ with $A$ being the set of lamps that are currently
lit. The lamplighter can make one of two types of moves: he can either
move to a neighbouring vertex of $G$ without changing the configuration
of lamps that are lit, or he can change the state of the lamp at
vertex $x$ and stay at vertex $x$.  We will refer to these as
\emph{horizontal} and \emph{vertical }moves,
respectively. (See~Figure~\ref{fig:lamplighter}.)

\begin{figure}[ht]
  \label{fig:lamplighter}
  \caption{Horizontal moves within fibers and vertical moves between
    fibers of the lamplighter graph}
  \vskip .5cm
  \includegraphics[scale=.5]{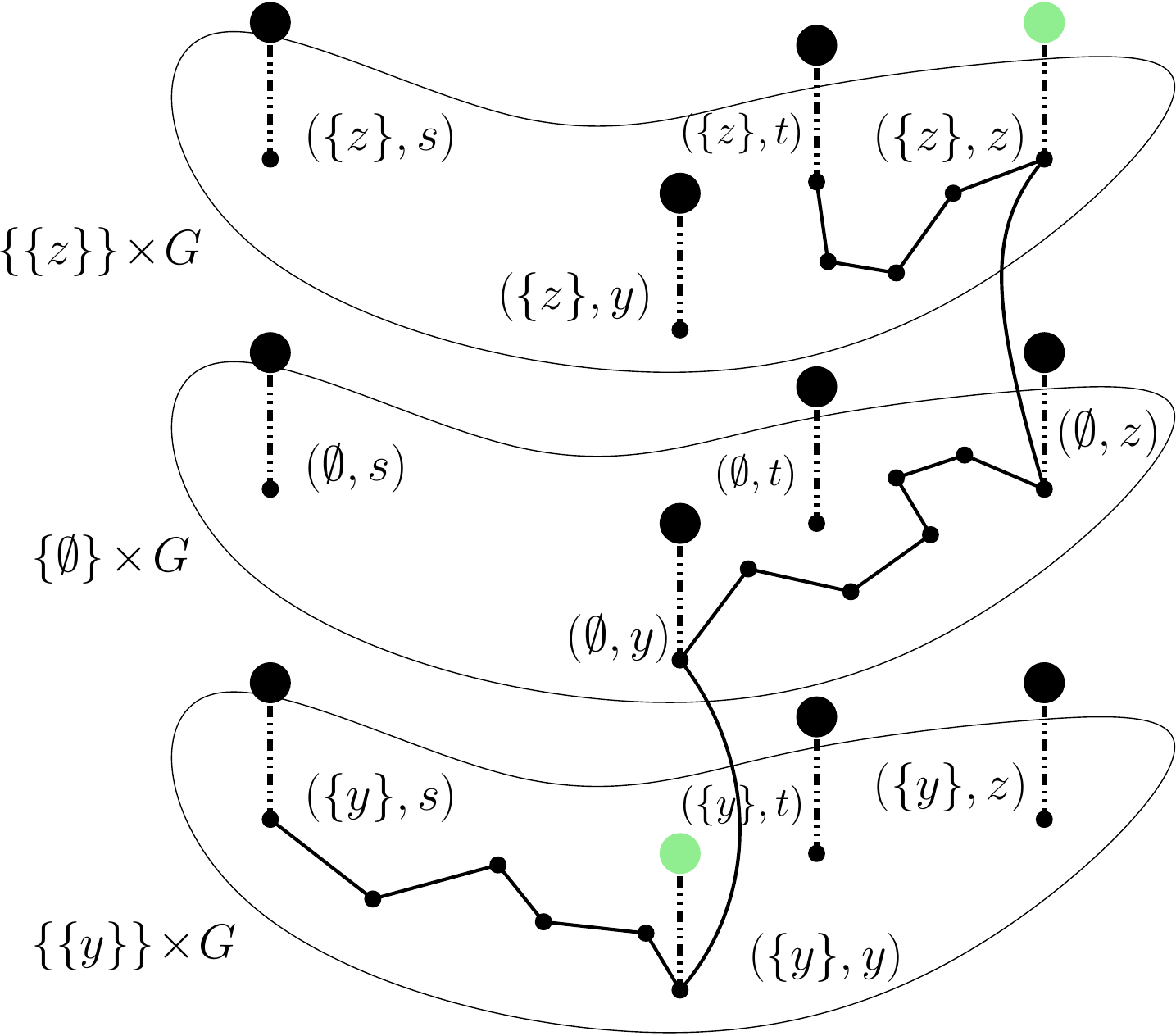}
\end{figure}

Other models with different available moves can also be considered,
such as the move-and-switch/move model or the like. Note that just as
different finite generating sets of a group lead to bi-Lipschitzly
equivalent Cayley graphs, it is also easy to verify whether two models
of lamplighter graphs are bi-Lipschitzly equivalent. Here we are
talking about graphs as metric spaces with the geodesic distance. We
will recall this and other standard graph-theoretic notions in
Section~\ref{sec:preliminaries}.

Our first main result is about lamplighter graphs over arbitrary trees.

\begin{mainthm}
  \label{mainthm:A}
  Let $T$ be a (non-empty) tree. Then there is a set $I$ such that
  $\la(T)$ bi-Lipschitzly embeds into the Hamming cube $\ham_I$. More
  precisely, there exists a map $f\colon \la(T)\to \ham_I$ such that
  \begin{equation}
    \tfrac12\cdot d_{\la(T)}(x,y)\leq d_{\ham}\big(f(x),f(y)\big)\leq
    3\cdot d_{\la(T)}(x,y)
  \end{equation} 
  for all $x,y\in \la(T)$. Moreover, if $T$ is finite or countable,
  then $I$ can also be chosen to be finite or countable,
  respectively.
\end{mainthm}

It follows from Theorem~\ref{mainthm:A} that the lamplighter graph over a countable
tree bi-Lipschitzly embeds into $\ell_1$. In particular, this applies to
the lamplighter group of a finitely generated free group; as we
mentioned in the Introduction, this result also follows from more
general results by Cornulier, Stadler and Valette ~\cite{csv:12}. Unlike ~\cite{csv:12} that
relies on geometric group-theoretic arguments, our approach is based on
elementary metric techniques.

Our second main result is a technical structural result
(Theorem~\ref{thm:coalescence-lemma}) that relates the geometry of the
lamplighter graph over the vertex-coalescence of two graphs with the
geometry of the coalesced components. By combining this
structural result together with several embedding results that are
discussed in Section~\ref{sec:5}, we extend the metric characterizations of  superreflexivity in terms of lamplighter groups of~\cite{RO2018} to characterizations in terms of lamplighter graphs over graphs that
are built by coalescing several copies of elementary graphs such as
cycles or paths. In order to state our next result, we recall some
basic definitions from metric geometry. Let $(M,d_M)$ and $(N,d_N)$ be
two metric spaces. A map $f\colon M\to N$ is called a
\emph{bi-Lipschitz embedding }if there exist $s>0$ and $D\geq 1$ such
that for all $u,v\in M$, we have
\begin{equation}
  \label{eqn:1.1_1}
  s\cdot d_M(u,v)\leq d_N\big(f(u),f(v)\big)\leq D\cdot s\cdot d_M(u,v)\ .
\end{equation}
The \emph{distortion }$\dist(f)$ of a bi-Lipschitz embedding $f$ is
given by
\[
\dist(f)=\sup_{u\neq v}\frac{d_N\big(f(u),f(v)\big)}{d_M(u,v)}\cdot
\sup_{u\neq v}\frac{d_M(u,v)}{d_N\big(f(u),f(v)\big)}\ .
\]
As usual,
\[
c_{N}(M)=\inf\big\{\dist(f) :\,f\colon M\to N \text{ is a bi-Lipschitz
  embedding}\big\}
\]
denotes the \emph{$N$-distortion of $M$}. If there is no bi-Lipschitz
embedding from $M$ into $N$, then we set $c_{N}(M)=\infty$. A sequence
$(M_k)_{k\in\bn}$ of metric spaces is said to
\emph{equi-bi-Lipschitzly embed }into a metric space $N$ if
$\sup_{k\in\bn}c_N(M_k)<\infty$.

Denote by $\st_{n,k}$ the \emph{star graph }with $n$ branches of
length $k$, and by $\rose_{n,k}$ the \emph{rose graph }whose $n$
leaves are $k$-cycles (see Figure~\ref{fig:star-rose}; definitions
will be given in Section~\ref{sec:amalgamation}).

\begin{figure}[h]
  \caption{The star graph $\st_{8,4}$ and the rose graph $\rose_{4,11}$}
  \label{fig:star-rose}
  \vskip .3cm
  \includegraphics[scale=0.4]{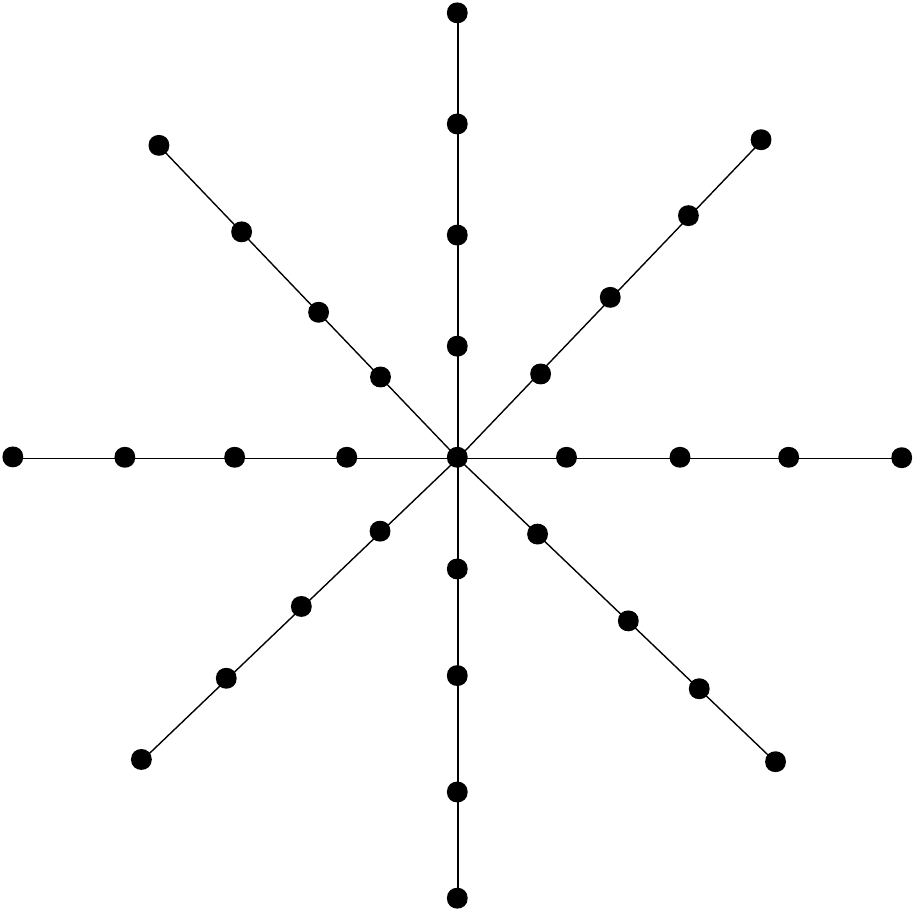}
  \hskip 1cm
  \includegraphics[scale=0.35]{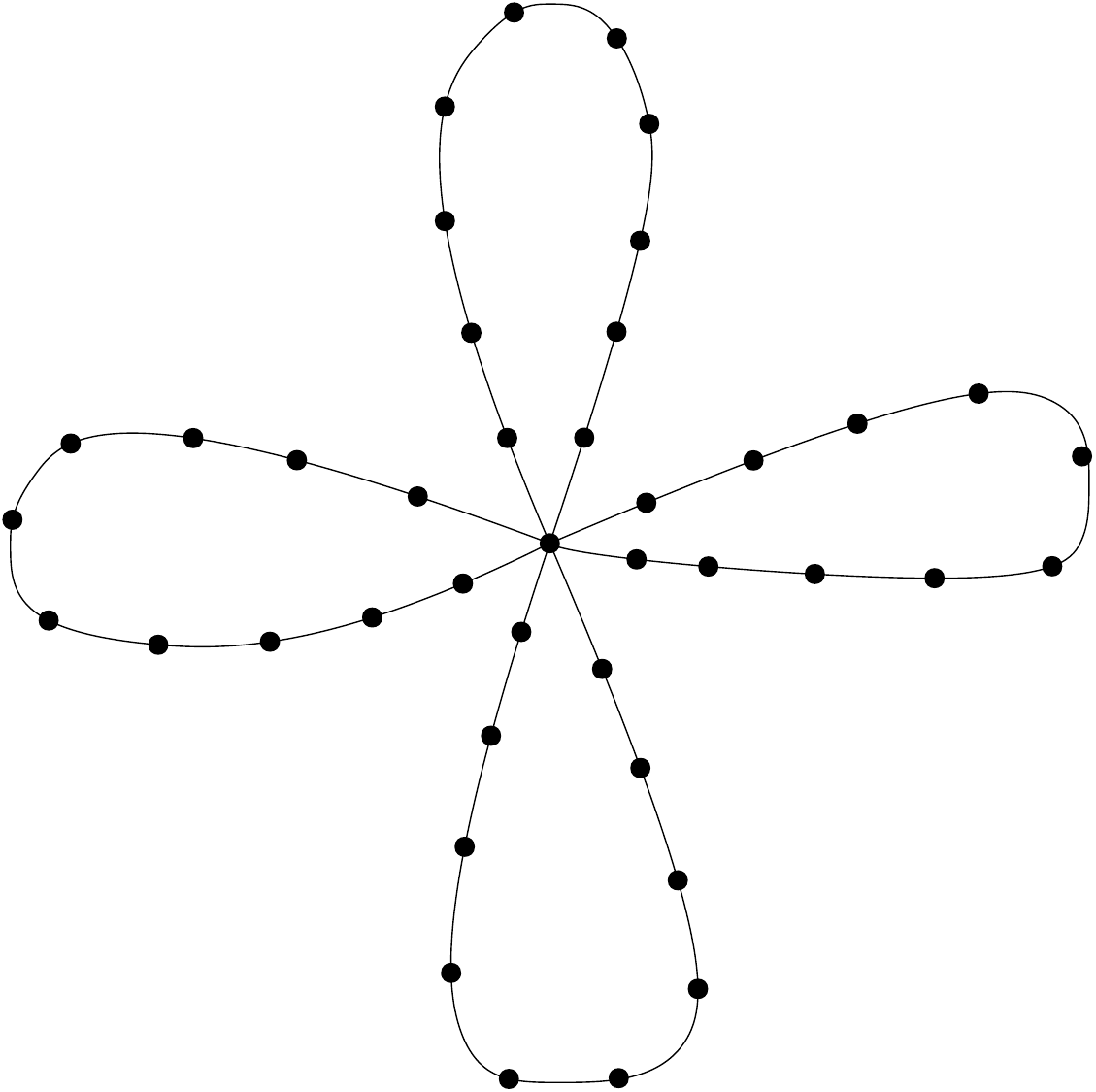}
\end{figure}

While these graphs can be easily embedded into every
finite-dimensional Banach space of a sufficiently large dimension, it
is far from being the case for lamplighter graphs over them.

\begin{mainthm}
  \label{mainthm:B}
  Let $Y$ be a Banach space and $n\in\bn$. The following assertions
  are equivalent.
  \begin{mylist}{(iii)}
  \item
    $Y$ is superreflexive;
  \item
    $\ds\sup_{k\in\bn}c_Y\big(\la(\st_{n,k})\big)=\infty$;
  \item
    $\ds\sup_{k\in\bn}c_Y\big(\la(\rose_{n,k})\big)=\infty$.
  \end{mylist}
\end{mainthm}

If $(M_k)_{k\in\bn}$ and $(N_k)_{k\in\bn}$ are sequences of metric
spaces, we say that \emph{$(M_k)_{k\in\bn}$ equi-bi-Lipschitzly embeds
  into $(N_k)_{k\in\bn}$}, or \emph{$(N_k)_{k\in\bn}$
  equi-bi-Lipschitzly contains $(M_k)_{k\in\bn}$}, if
$\sup_k \inf_\ell c_{N_\ell} (M_k)<\infty$, or equivalently, if there is a
$C>0$ such that for all $k\in\bn$ there exists $\ell\in\bn$ such that
$M_k$  bi-Lipschitzly embeds into $N_\ell$ with distortion at
most~$C$. We say that $(M_k)_{k\in\bn}$ and $(N_k)_{k\in\bn}$ are
\emph{Lipschitz-comparable} if $(M_k)_{k\in\bn}$ equi-bi-Lipschitzly
embeds into $(N_k)_{k\in\bn}$ and $(N_k)_{k\in\bn}$
equi-bi-Lipschitzly embeds into $(M_k)_{k\in\bn}$. In
Section~\ref{sec:5} we prove that the lamplighter graph over $\com_k$,
the complete graph with $k$ vertices, contains a bi-Lipschitz copy of
the $k$-dimensional Hamming cube $\ham_k$ with distortion independent
of $k$.  Together with Theorem~\ref{mainthm:A}, it follows that the
geometry of lamplighter graphs over complete graphs or over binary
trees is essentially the same as the geometry of the Hamming cubes.

\begin{mainthm}
  \label{mainthm:C}
  The sequences $\big(\la(\com_k)\big)_{k\in\bn}$,
  $\big(\la(\bin_k)\big)_{k\in\bn}$
  and $\big(\ham_k\big)_{k\in\bn}$ are pairwise Lipschitz-comparable.
\end{mainthm}

Theorem \ref{mainthm:C} has an important consequence regarding characterizations of the notion of trivial Bourgain-Milman-Wolfson type \cite{bmw:86} (BMW-type in short). In 1986, Bourgain, Milman, and Wolfson showed that a metric space $Y$ has trivial BMW-type if and only if $\sup_{k\in\bn}c_Y\big(\ham_{k})<\infty$. This result is a nonlinear analogue of the Maurey-Pisier theorem for trivial type.
The Hamming cube $\ham_k$ is a $k$-regular graph and thus $\big(\ham_k\big)_{k\in\bn}$ is a sequence of graphs with unbounded degree. The notion of BMW-type comes from the local theory of Banach spaces and a natural question is whether trivial BMW-type can be characterized as above using a sequence of graphs $(G_k)_{k\in\bn}$ with uniformly bounded degree. For Banach spaces, Ostrovskii \cite{ostrovskii:11} answered this question positively and it is not difficult to see that the sequence of graphs with maximum degree 3 from \cite{ostrovskii:11} is actually Lipschitz-comparable to the sequence of Hamming cubes and thus also settles the question for arbitrary metric spaces. Since every graph in the sequence $\big(\la(\bin_k)\big)_{k\in\bn}$ has maximum degree $4$, Theorem \ref{mainthm:C} also provides a sought after sequence $(G_k)_{k\in\bn}$. 
\begin{corD}
  \label{cor:D}
  Let $(Y,d_Y)$ be a metric space. Then,
\[Y \textrm{ has trivial BMW-type if and only if }\ds\sup_{k\in\bn}c_Y\big(\la(\bin_{k})\big)<\infty.\]
\end{corD}

Note that the similar question for the nonlinear analogue of the Maurey-Pisier theorem for trivial cotype, obtained by Mendel and Naor \cite{mendel-naor:08}, has a simple solution for arbitrary metric spaces (see \cite{ostrovskii:11} for a proof and a discussion of these questions). 

\section*{Acknowledgment}
We thank the referees for their careful reading and the many suggestions that helped improve the exposition and clarity of the paper.
\section{Preliminaries on lamplighter graphs}
\label{sec:preliminaries}

We shall use standard graph theory terminology as can be found
in~\cite{bollobas:98}. In particular, a graph $G$ is a pair $(V,E)$
where $V=V(G)$ is an arbitrary set (the set of vertices) and $E=E(G)$
is the set of edges, \ie a set consisting of some unordered pairs of
distinct vertices. (So edges are not directed and there are no
multiple edges or loops.) We shall often write $x\in G$ instead of
$x\in V$ for a vertex $x$. The edge connecting distinct vertices $x$ 
and $y$ is simply denoted by $xy$ (which is the same as $yx$). A
\emph{walk in }$G$ is a finite sequence
$w=(x_0,x_1,\dots,x_n)$ of vertices of $G$ with $n\geq 0$ such that
$x_{i-1}x_i$ is an edge of $G$ for all $1\leq i\leq n$. We call $w$ a
\emph{walk from $x=x_0$ to $y=x_n$ }and call $n$ the \emph{length of
  $w$}. If $w$ has no repetition of vertices other than the first and
last vertices, \ie if $x_i\neq x_j$ whenever $1<j-i<n$, then $w$ is
called a \emph{path (from $x$ to $y$)}. If $w$ is a walk and $x_r=x_s$
for some $r,s$ with $1<s-r<n$, then
$(x_0,\dots,x_{r-1},x_s,x_{s+1},\dots,x_n)$ is a strictly shorter walk
from $x$ to $y$. It follows that if $w'$ is a subsequence of $w$ of
minimal length such that $w'$ is a walk from $x$ to $y$, then $w'$ is
in fact a path. We say that the graph $G$ is \emph{connected }if any
two vertices are connected in $G$ by a walk (or, equivalently, by a
path). 

A connected graph $G$ becomes a metric space in a natural way. For
vertices $x$ and $y$ of $G$, we denote by $d_G(x,y)$ (or sometimes
simply by $d(x,y)$) the length of a shortest path in $G$ (called a
\emph{geodesic}) from $x$ to $y$. It is easy to verify that $d_G$ is a
metric. An important example for us are \emph{Hamming cubes}. For an
arbitrary set $I$, the Hamming cube $\ham_I$ has vertex set
$\{0,1\}^{(I)}$ consisting of all functions $\vare\colon I\to\{0,1\}$
with finite support, \ie the set $\{i\in I:\,\vare_i=1\}$ is a finite
subset of $I$. Two vertices $\vare$ and $\delta$ are joined by an edge
if and only if they differ in exactly one coordinate, \ie there is a
unique $i\in I$ with $\vare_i\neq\delta_i$. The graph distance on
$\ham_I$, denoted $d_{\ham}$ and referred to as the \emph{Hamming
  metric}, is the $\ell_1$-metric given by
\[
d_{\ham}(\vare,\delta)=\sum_{i\in I} \abs{\vare_i-\delta_i}\ .
\]
We shall often identify $\{0,1\}^{(I)}$ with the set of all finite
subsets of $I$. Under this identification, the Hamming metric becomes
the symmetric difference metric given by
$d_{\ham}(A,B)=\abs{A\symdif B}$ for finite subsets $A$ and $B$ of $I$.

\subsection{A closed formula for the lamplighter graph metric}
Let us recall the definition of the lamplighter graph $\la(G)$ of a
graph $G$. The vertex set of $\la(G)$ consists of all pairs $(A,x)$
with $x\in G$ and $A$ a finite subset of $G$. Two vertices $(A,x)$ and
$(B,y)$  are joined by an edge in $\la(G)$ if and only if \emph{either
}$A=B$ and $xy$ is an edge in $G$ \emph{or }$x=y$ and
$A\symdif B=\{x\}$, and these correspond, respectively, to horizontal
and vertical moves by the lamplighter.

It is clear that if $\la(G)$ is connected, then so is $G$. Indeed, the
horizontal moves in a path in $\la(G)$ from $(\emptyset,x)$ to
$(\emptyset,y)$ correspond to a path from $x$ to $y$ in $G$. The
converse also holds and its proof yields a formula for the graph
metric of $\la(G)$ given in
Proposition~\ref{prop:graph-metric-in-lamplighter} below. Computing 
the distance in $\la(G)$ boils down to the problem of finding a
shortest walk in $G$ from a vertex $x$ to another vertex $y$ that
visits all vertices in a given subset $C$ of $G$. This is a well known
and famous problem, the \emph{travelling salesman problem }for $G$. We
shall denote by $\tsp_G(x,C,y)$ the length of a solution to this
problem, \ie the least $n\geq 0$ for which there is a walk
$(x_0,x_1,\dots,x_n)$ from $x=x_0$ to $y=x_n$ such that
$C\subset\{x_0,x_1,\dots,x_n\}$.

\begin{prop}
  \label{prop:graph-metric-in-lamplighter}
  Let $G$ be a  connected graph. Then the lamplighter graph
  $\la(G)$ is also connected with graph metric given by
  \begin{equation}
    \label{eq:graph-metric-in-lamplighter}
    d_{\la(G)}\big((A,x),(B,y)\big)= \tsp_G(x,A\symdif B,y) +
    \abs{A\symdif B}\ .
  \end{equation}
\end{prop}
\begin{proof}
  Let us fix vertices $(A,x)$ and $(B,y)$. The lamplighter clearly
  needs at least $\abs{A\symdif B}$ vertical moves in getting from
  $(A,x)$ to $(B,y)$ in order to switch all lamps in $A\setminus B$
  off and to lit all lamps in $B\setminus A$. As the lamplighter can
  only alter the state of the lamp at the vertex he is currently at,
  his horizontal moves must visit all vertices in $A\symdif B$ while
  travelling from $x$ to $y$. Thus, the right-hand side in the
  expression above is a lower bound for the distance. It is easy to see
  that this lower bound is attained. Indeed, let
  $n=\tsp_G(x,A\symdif B,y)$ and let $w=(x_0,x_1,\dots,x_n)$ be a walk
  in $G$ from $x=x_0$ to $y=x_n$ such that
  $A\symdif B\subset\{x_0,x_1,\dots,x_n\}$. Let $m=\abs{A\symdif B}$
  and let $0\leq i_1<i_2<\dots <i_m\leq n$ be such that
  $A\symdif B=\{x_{i_1},\dots,x_{i_m}\}$. Set $i_{m+1}=n$. Now
  consider the following path of length $m+n$ in $\la(G)$ from
  $(A,x)$ to $(B,y)$. Start with horizontal moves $(A,x_i)$,
  $0\leq i\leq i_1$, from $(A,x)$ to $(A,x_{i_1})$. Having reached the
  vertex $\big(A\symdif \{x_{i_1},\dots,x_{i_{j-1}}\},x_{i_j}\big)$
  for some $1\leq j\leq m$, make the vertical move
  $\big(A\symdif\{x_{i_1},\dots,x_{i_j}\},x_{i_j}\big)$ followed by
  horizontal moves $\big(A\symdif\{x_{i_1},\dots,x_{i_j}\},x_i\big)$
  for $i_j<i\leq i_{j+1}$. These moves end at the vertex
  $\big(A\symdif\{x_{i_1},\dots,x_{i_j}\},x_{i_{j+1}}\big)$ which
  becomes $(B,y)$ when $j=m$.
\end{proof}

In general, the Travelling Salesman Problem is NP-hard. However, for
some graphs it is possible to find explicit algorithms. We present one
such algorithm for trees in Section~\ref{sec:3.1}. This is essentially
a pre-order traversal algorithm that also accounts for backtracking.

\subsection{A first example: the lamplighter graph over a path}
It was shown, amongst other things, in~\cite{stein-taback:12} that the
group $\bz_q\wr \bz$ (with a specific generating set) embeds
bi-Lipschitzly with distortion at most~$4$ in the Cartesian product of
two infinite $(q+1)$-regular trees. Their proof is based on a lengthy
and tricky computation of the word-length of certain group
elements. We provide a simpler proof of the finite analogue for the
case $q=2$ in the purely graph-theoretic
context. We replace $\bz$ with $\pa_k$, the path of length $k$, which
has vertices $v_0, v_1,\dots,v_k$ and edges $v_{i-1}v_i$,
$1\leq i\leq k$, and replace the lamplighter group $\bz_2\wr\bz$ with
the lamplighter graph $\la(\pa_k)$. After the proof we explain how
our argument extends to the infinite case, and hence shows the result
for $\bz_2\wr\bz$ just mentioned.

We first describe \emph{the binary tree $\bin_k$ of height $k$}, and
introduce some notation. The vertex set of $\bin_k$ is
$\bigcup_{i=0}^k \{0,1\}^i$. Let $\delta=(\delta_1,\dots,\delta_m)$ and
$\vare=(\vare_1,\dots,\vare_n)$ be vertices of $\bin_k$. We write
$\delta\prec \vare$ if $m<n$ and $\delta_i=\vare_i$ for
$1\leq i\leq m$. Then $\delta\vare$ is an edge of $\bin_k$ if and only
if $\abs{m-n}=1$ and either $\delta\prec \vare$ or
$\vare\prec \delta$. We will write $\delta\preccurlyeq \vare$ if
$\delta=\vare$ or $\delta\prec \vare$. We define the length of
$\delta$ to be $\abs{\delta}=m$. If $m\geq 1$, we let
$\delta'=(\delta_1,\dots,\delta_{m-1})$. Note that if
$\abs{\delta}\leq \abs{\vare}$, then $\delta\vare$ is an edge of
$\bin_k$ if and only if $\delta=\vare'$. The unique vertex of length
zero will be denoted by $\emptyset$ and the graph-metric by $d_{\bin}$
regardless of the value of~$k$.

Given graphs $G$ and $H$, the
\emph{Cartesian product graph }$G\cprod H$ of $G$ and $H$ has vertex
set $V(G)\times V(H)$, and vertices $(x,y)$ and $(v,z)$ are joined by
an edge if and only if \emph{either }$x=v$ and $yz$ is an edge in $H$
\emph{or }$y=z$ and $xv$ is an edge in $G$. Observe that the graph
metric on the Cartesian product is given by
\[
d_{\square}\big((x,y),(v,z)\big)=d_{G}(x,v)+d_{H}(y,z)\ .
\]
Note that the Hamming cube $\ham_n=\ham_{\{1,\dots,n\}}$ is the
$n$-fold Cartesian product graph $\pa_1\cprod \dots \cprod \pa_1$.

\begin{prop}
  \label{prop:path-into-trees}
  Let $k\in\bn$. There exists a map
  $f\colon \la(\pa_k)\to \bin_{k+1}\cprod \bin_{k+1}$ such that for
  all $x,y\in\la(\pa_k)$ we have
  \begin{equation*}
    \tfrac23\cdot d_{\la(\pa_k)}(x,y)\leq
    d_{\square}\big(f(x),f(y)\big) \leq 2\cdot d_{\la(\pa_k)}(x,y)\ .
  \end{equation*}
\end{prop}

\begin{proof}
  For $(A,v_m)\in \la(\pa_k)$ let
  $f(A,v_m)=\big((\vare^A_i)_{i=1}^m,(\vare^A_{k+1-i})_{i=0}^{k-m}\big)$
  where
  \[
  \vare^A_i=%
  \begin{cases}
    1 & \text{ if } v_{i-1}\in A\ ,\\
    0 & \text{ if } v_{i-1}\notin A\ .
  \end{cases}
  \]
  Let $(A,v_m), (B,v_n)\in \la(\pa_k)$ and assume without loss of
  generality that $m\leq n$. If $A\symdif B=\emptyset$, then $A=B$,
  $\vare^A=\vare^B$ and
  \begin{equation}
    \label{eq:4}
    \begin{aligned}
      d_{\square} \big(f(A,v_m),&f(B,v_n)\big)\\
      &=d_{\bin}\big((\vare^A_i)_{i=1}^m,(\vare^B_i)_{i=1}^n)\big) +
      d_{\bin}
      \big((\vare^A_{k+1-i})_{i=0}^{k-m},(\vare^B_{k+1-i})_{i=0}^{k-n}\big)\\
      &=n-m+(k-m)-(k-n)\\  
      &=2\cdot d_{\la(\pa_k)}\big((A,v_m),(B,v_n)\big)\ .
    \end{aligned}
  \end{equation}
  If $v_m=v_n$ and $A\symdif B=\{v_m\}$, then $\vare^A_i=\vare^B_i$ if
  and only if $i\neq m+1$, and hence
  \begin{equation}
    \label{eq:8}
    \begin{aligned}
      d_\square\big(f(A,v_m),&f(B,v_n)\big)\\
      &=d_{\bin}
      \big((\vare^A_i)_{i=1}^m,(\vare^{B}_i)_{i=1}^m\big)
      + d_{\bin}\big((\vare^A_{k+1-i})_{i=0}^{k-m},
      (\vare^{B}_{k+1-i})_{i=0}^{k-m}\big)\\ 
      &=0+2=2\ .
    \end{aligned}
  \end{equation}
  Since $d_{\la(\pa_k)}$ is a graph metric, it is sufficient to
  estimate the Lipschitz constant on adjacent vertices, and it follows
  from~\eqref{eq:4} and~\eqref{eq:8} that $f$ is $2$-Lipschitz.

  Assume now that $A\symdif B\neq \emptyset$. Set
  $\ell=\min\{i:\,v_{i-1}\in A\symdif B\}=\min\big\{i:\,\vare^A_i\neq\vare^B_i\big\}$
  and
  $r=\max\{i:\,v_{i-1}\in A\symdif B\}=\max\big\{i:\,\vare^A_i\neq\vare^B_i\big\}$.
  From the definition of $\ell$ and $r$ it follows that
  \begin{equation}
    \label{eq:left}
    \begin{aligned}
      d_{\bin} \big((\vare^A_i)_{i=1}^m,(\vare^B_i)_{i=1}^n\big) &= %
      \begin{cases}
        n-m & \text{if } \ell>m\ ,\\
        m-(\ell-1) +n-(\ell-1) & \text{if } \ell\leq m\ ,
      \end{cases}\\[2ex]
      &=%
      \begin{cases}
        n-m & \text{if } \ell>m\ ,\\
        m+n+2-2\ell & \text{if } \ell\leq m\ .
      \end{cases}
    \end{aligned}
  \end{equation}
  and
  \begin{equation}
    \label{eq:right}
    \begin{aligned}
      d_{\bin}\big((\vare^A_{k+1-i})_{i=0}^{k-m},&(\vare^B_{k+1-i})_{i=0}^{k-n}\big)\\
      &=%
      \begin{cases}
        (n+1)-(m+1) &\text{if } r\leq  n\ ,\\
        (r+1)-(m+1)+(r+1)-(n+1) &\text{if } r>n\ ,
      \end{cases}\\[2ex]
      &=%
      \begin{cases}
        n-m & \text{if } r\leq n\ ,\\
        2r-m-n & \text{if } r>n\ .
      \end{cases}
    \end{aligned}
  \end{equation}
  Obtaining a lower bound on $d_\square\big(f(A,v_m),f(B,v_n)\big)$
  using~\eqref{eq:left} and~\eqref{eq:right} naturally splits into
  four cases. In all cases we will use the estimate
  $\abs{A\symdif B}\leq r-\ell+1$.
  \begin{mylist}{}
  \item[\underline{Case 1:}]
    $\ell\leq m$ and $r\leq n$. In this case
    $\tsp_{\pa_k}(v_m,A\symdif B,v_n)=m+n+2-2\ell$ as the salesman
    moves from $v_m$ to $v_{\ell-1}$ and then to $v_n$. We then get
    \begin{multline*}
      d_{\la(\pa_k)}\big((A,v_m),(B,v_n)\big) \leq m+n+2-2\ell +
      r-\ell+1 = r+m+n-3(\ell-1)\\
      \leq 3(n-\ell+1) =\tfrac32 \cdot
      d_\square\big(f(A,v_m),f(B,v_n)\big)\ .
    \end{multline*} 
  \item[\underline{Case 2:}]
    $\ell\leq m$ and $r>n$. In this case
    $\tsp_{\pa_k}(v_m,A\symdif B,v_n)=m-n+2r-2\ell$ as the salesman
    moves from $v_m$ to $v_{\ell-1}$, then to $v_{r-1}$ and finally to
    $v_n$. Thus,
    \begin{multline*}
      d_{\la(\pa_k)}\big((A,v_m),(B,v_n)\big) \leq m-n+2r-2\ell +
      r-\ell+1 = m-n-2+3(r-\ell+1)\\
      \leq 3(r-\ell+1) =\tfrac32 \cdot
      d_\square\big(f(A,v_m),f(B,v_n)\big)\ .
    \end{multline*} 
  \item[\underline{Case 3:}]
    $\ell>m$ and $r\leq n$. Then $\tsp_{\pa_k}(v_m,A\symdif
    B,v_n)=n-m$ as the optimal walk for the salesman is from $v_m$ to
    $v_n$. Therefore,
    \begin{multline*}
      d_{\la(\pa_k)}\big((A,v_m),(B,v_n)\big) \leq n-m+
      r-\ell+1 \leq n+r-2m\\
      \leq 2(n-m)=
      d_\square\big(f(A,v_m),f(B,v_n)\big)\ .
    \end{multline*} 
  \item[\underline{Case 4:}]
    $\ell> m$ and $r>n$. In this range
    $\tsp_{\pa_k}(v_m,A\symdif B,v_n)=2r-2-m-n$ as the salesman moves
    from $v_m$ to $v_{r-1}$ and to $v_n$. Using~\eqref{eq:left}
    and~\eqref{eq:right} for the last time, we get
    \begin{multline*}
      d_{\la(\pa_k)}\big((A,v_m),(B,v_n)\big) \leq 2r-2-m-n +
      r-\ell+1 = 3r-2-m-n-(\ell-1)\\
      \leq 3r-3m =\tfrac32 \cdot
      d_\square\big(f(A,v_m),f(B,v_n)\big)\ .
    \end{multline*} 

  \end{mylist}
\end{proof}

\begin{rem}
  Let $\mathrm{Z}$ denote the double-infinite path. This graph has vertex
  set $\bz$ and edges between consecutive integers. Let $\mathrm{T}_3$
  be the $3$-regular
  (infinite) tree. A description of $\mathrm{T}_3$ is as follows. For
  $n\in\bz$ denote by $\bz_{\leq n}$ the initial segment
  $\{m\in\bz:\,m\leq n\}$ of $\bz$. Then $\mathrm{T}_3$ has vertex set
  \[
  \big\{ \vare\colon\bz_{\leq n}\to\{0,1\}:\,n\in\bz,\ \vare \text{
    has finite support} \big\}
  \]
  and vertices $(\delta_i)_{i=-\infty}^m$ and
  $(\vare_i)_{i=-\infty}^n$ with $m\leq n$ are joined by an edge if
  and only if $n=m+1$ and $\delta_i=\vare_i$ for all $i\leq m$. An
  almost identical argument as the one used in the proof above shows
  that the map $f\colon\la(\mathrm{Z})\to\mathrm{T}_3\cprod\mathrm{T}_3$
  defined by
  $f(A,n)=\big((\vare^A_i)_{i=-\infty}^n,(\vare^A_{-i})_{i=-\infty}^{-n-1}\big)$
  has distortion at most~3, where $\vare^A$ denotes the indicator
  function of the finite subset $A$ of $\bz$. It is clear that
  $\la(\mathrm{Z})$ is isometric to $\bz_2\wr\bz$ with respect to a
  suitable set of generators.
\end{rem}

\subsection{Lamplighter graphs vs lamplighter groups}
We conclude this section by making precise the connection between
lamplighter graphs and lamplighter groups. As previously mentioned,
the lamplighter group of a group $\Gamma$ is the (restricted) wreath
product $\bz_2\wr \Gamma$. This can be thought of as the set of all
pairs $(A,x)$ with $A$ a finite subset of $\Gamma$ and $x\in \Gamma$
with multiplication defined by
\[
(A,x)\cdot (B,y)=(A\symdif xB,xy)
\]
where $xB=\{ xb:\,b\in B\}$. Now assume that $\Gamma$ is generated by
$S\subset \Gamma$. We assume that the identity $e\notin S$ and that
$x^{-1}\in S$ whenever $x\in S$. The \emph{(right) Cayley graph
}$\cay(\Gamma,S)$ of $\Gamma$ with respect to $S$ has vertex set
$\Gamma$, and $x,y\in \Gamma$ are joined by an edge if and only if
$y^{-1}x\in S$. Since $S$ generates $\Gamma$, it follows that
$\cay(\Gamma,S)$ is connected. It is easy to verify that
\[
S'=\big\{ (\emptyset,s):\, s\in S\big\}\cup \big\{ \big(\{e\},e\big) \big\}
\]
generates $\bz_2\wr \Gamma$. Moreover, the Cayley graph
$\cay(\bz_2\wr \Gamma,S')$ is the lamplighter graph
$\la\big(\cay(\Gamma,S)\big)$.

\begin{rem}
  It is possible to define the wreath product of graphs which
  generalizes the notion of wreath product of groups. Let
  $G=(V_G, E_G)$ and  $H=(V_H,E_H)$ be two graphs, and let $v_0$ be a
  distinguished point in $V_H$. A function $f\colon V_G\to V_H$ is
  called finitely supported if $f(v)=v_0$ for all but finitely many
  $v\in V_G$. The \emph{wreath product $H\wr G$ of $H$ with $G$ }is
  the graph with vertex set
  \[
  V_H^{(V_G)}\times V_G=\big\{(f,v):\,f\colon V_G\to V_H \text{
    finitely supported, }v\in V_G\}\ ,
  \]
  and two vertices $(f,x)$ and $(g,y)$ are connected by an edge if and
  only if \emph{either }$f=g$ and $xy$ is an edge in $G$ \emph{or
  }$x=y$, $f(v)=g(v)$ for every $v\in G\setminus\{x\}$, and $f(x)g(x)$
  is an edge in $H$. As in the special case above, it is easy to
  verify that if $G$ and $H$ are Cayley graphs of groups $\Gamma$ and
  $\Delta$, respectively, then $H\wr G$ is the Cayley graph of
  $\Delta\wr\Gamma$ with respect to a suitable generating set.
  
In this paper we are concerned with lamplighter graphs. It is not too hard to verify that some of our results extend fairly easily to more general wreath products. As wreath products with $\bz_2$ are of greatest interest, we prefer to state and
prove our results only for such products.  One
justification for concentrating on lamplighter graphs instead of more
general wreath products is as follows. Any finite graph $H$ is Lipschitz
isomorphic to the complete graph $S$ on the vertex set of $H$ with
distortion $D$ depending on $H$. This naturally induces a Lipschitz
isomorphism between $H\wr G$ and $S\wr G$ with distortion at most $D$
for any graph $G$. An argument similar to \cite[Lemma 2.1]{naor-peres:11} shows that the bi-Lipschitz embeddability into $L_p$ of
$S\wr G$ and $\la(G)$ are the same up to universal constants.
\end{rem}

\section{Embeddability of lamplighter graphs over trees into Hamming cubes}

\subsection{The Travelling Salesman Problem for trees}\label{sec:3.1}

A \emph{tree }is a connected acyclic graph, \ie a connected graph in
which there is no path $(x_0,\dots,x_n)$ with $n\geq 3$ and
$x_0=x_n$. Equivalently, a tree is a graph such that for any two
vertices $x$ and $y$ there is a unique path from $x$ to $y$. \Eg every
binary tree is a tree.

We now fix a tree $T$ for the rest of this section. For vertices
$x,y\in T$ we denote by $p(x,y)$ the unique path in $T$ from $x$ to
$y$. If $p(x,y)=(x_0,x_1,\dots,x_n)$, then we let $p_i(x,y)=x_i$ for
$0\leq i\leq n$, and we also let
$[x,y]=\{x_{i-1}x_i:\,1\leq i\leq n\}$ be the set of edges on the path
$p(x,y)$. By definition of a path, every edge in $[x,y]$ occurs
exactly once, and so $\bigabs{[x,y]}=d_{T}(x,y)$. It is also clear
that if $p(x,y)=(x_0,x_1,\dots,x_n)$, then
$p(y,x)=(x_n,x_{n-1},\dots,x_1,x_0)$, and hence $[x,y]=[y,x]$.

For $x\in T$ and for $A\subset T$ we let
$[x,A]=\bigcup_{a\in A} [x,a]$. Note that if $A$ is finite, then so is
$[x,A]$. We are now ready to provide a closed formula for the
Travelling Salesman Problem on a tree.

\begin{thm}
  \label{thm:tsp-in-tree}
  For $x,y\in T$ and a finite $A\subset T$, we have
  \[
  \tsp_{T}(x,A,y)=2\bigabs{[x,A]\setminus [x,y]}+\bigabs{[x,y]}\ .
  \]
\end{thm}

We begin the proof with a couple of simple lemmas.

\begin{lem}
  \label{lem:off-road-edges}
  Let $x,y,a$ be vertices of $T$. Then
  $[x,a]\setminus [x,y]=[y,a]\setminus [x,y]$.
\end{lem}
\begin{proof}
  We may assume that $x,y,a$ are pairwise distinct, otherwise the
  result is clear. Let $p(x,y)=(x_0,x_1,\dots,x_m)$ and
  $p(x,a)=(y_0,y_1,\dots,y_n)$. Then $x_0=y_0=x$, $x_m=y$ and
  $y_n=a$. Choose $i$ maximal with $0\leq i\leq\min(m,n)$ such that
  $x_j=y_j$ for $0\leq j\leq i$. Then
  \[
  w=(y_n,y_{n-1},\dots,y_{i+1},y_i=x_i,x_{i+1},x_{i+2},\dots,x_m)
  \]
  is a walk from $a$ to $y$. We show that $w$ is in fact a path. If it
  is not, then we must have $i<\min(m,n)$ and $x_k=y_\ell$ for some
  $k,\ell$ with $i+1\leq k\leq m$ and $i+1\leq\ell\leq n$. Choosing
  $k$ minimal, we obtain a cycle
  \[
  p=(x_i,x_{i+1},\dots,x_k=y_\ell,y_{\ell-1},\dots,y_{i+1},y_i)
  \]
  in $T$. Indeed, $x_i=y_i$ and there is no other repetition of
  vertices by minimality of $k$. Thus, $p$ is a path from $x_i$ to
  $y_i$. Moreover, since $x_{i+1}\neq y_{i+1}$, either $k>i+1$ or
  $\ell>i+1$, and hence the length $(k-i)+(\ell-i)$ of $p$ is
  at least~3. This contradiction completes the proof that $w$ is a
  path, and so $p(a,y)=w$.

  Now let $e\in[x,a]\setminus[x,y]$. Since $e\in [x,a]$, we have
  $e=y_{j-1}y_j$ for some $1\leq j\leq n$, and since $e\notin [x,y]$,
  we must have $i<j$. It follows that $e$ is also on the path
  $w=p(a,y)$, \ie that $e\in[y,a]$. The inclusion
  $[x,a]\setminus [x,y]\subset [y,a]\setminus [x,y]$ follows, and the
  reverse inclusion holds by symmetry in $x,y$.
\end{proof}

The next lemma shows that any walk from $x$ to $y$ must travel through
every edge in the unique path from $x$ to $y$.

\begin{lem}
  \label{lem:unavoidable-edges}
  Let $x,y\in T$ and $w=(w_0,w_1,\dots,w_n)$ be a walk from $x$ to
  $y$. Then for every $e\in[x,y]$ there exists $1\leq j\leq n $ such
  that $e=w_{j-1}w_j$.
\end{lem}
\begin{proof}  
  Let $p(x,y)=(x_0,x_1,\dots,x_m)$. Then $e=x_{i-1}x_i$ for some
  $1\leq i\leq m$. We observed at the start of the
  previous section that in any graph, every walk between vertices
  contains a subsequence which is a path between the same vertices. It
  follows that $p(x,y)$ is a subsequence of $w$. Hence there is a
  maximal $j$, $1\leq j\leq n$, such that $w_{j-1}=x_{i-1}$. If
  $w_j=x_i$, then we are done. So let us assume $w_j\neq x_i$. Then
  $w_j\notin \{ x_k:\, i\leq k\leq m\}$ since otherwise we obtain a
  cycle in $T$. It follows that
  \[
  p(w_j,y)=(w_j, x_{i-1},x_i,\dots,x_m)\ ,
  \]
  which therefore must be a subsequence of the walk
  $(w_j,w_{j+1},\dots,w_n)$. In particular, $x_{i-1}=w_{k-1}$ for some
  $j<k\leq n$, which contradicts the maximality of $j$.
\end{proof}

We are now ready to prove the lower bound for the Travelling Salesman
Problem in $T$.

\begin{prop}
  \label{prop:tsp-in-tree-lower}
  For $x,y\in T$ and a finite $A\subset T$, we have
  \[
  \tsp_T(x,A,y)\geq 2\bigabs{[x,A]\setminus [x,y]}+\bigabs{[x,y]}\ .
  \]
\end{prop}
\begin{proof}
  Let $n=\tsp_T(x,A,y)$ and let $w=(w_0,w_1,\dots,w_n)$ be a walk from
  $x$ to $y$ such that $A\subset\{w_0,w_1,\dots,w_n\}$. By
  Lemma~\ref{lem:unavoidable-edges}, for every $e\in[x,y]$, there is
  at least one $j\in\{1,2,\dots,n\}$ such that $e=w_{j-1}w_j$.

  Now assume $e\in [x,A]\setminus[x,y]$. Then
  $e\in[x,a]\setminus[x,y]$ for some $a\in A$. Choose $i$ with
  $0\leq i\leq n$ and $a=w_i$. Then $(w_0,w_1,\dots,w_i)$ is a walk
  from $x$ to $a$, and hence by Lemma~\ref{lem:unavoidable-edges},
  $e=w_{j-1}w_j$ for some $1\leq j\leq i$. On the other hand, by
  Lemma~\ref{lem:off-road-edges} we also have $e\in[a,y]$. Since
  $(w_i,w_{i+1},\dots,w_n)$ is a walk from $a$ to $y$, it follows that
  $e=w_{k-1}w_k$ for some $i<k\leq n$. Since $j\neq k$, it follows
  that every edge in $[x,A]\setminus [x,y]$ appears at least twice in
  the walk $w$. The result follows.
\end{proof}

We next introduce some more notation. For $x\in T$ we denote by
$N_x$ the set of \emph{neighbours of $x$ }given by
$N_x=\{y\in T:\,xy\in E(T)\}$. For $y\in N_x$ we let
$T_{x,y}=\{z\in T:\,p_1(x,z)=y\}$, and for $y\in N_x$ and $A\subset T$
we let $A_{x,y}=A\cap T_{x,y}$. We now establish some simple
properties.

\begin{lem}
  \label{lem:subtrees}
  Fix $x\in T$ and $A\subset T$. We have then the following.
  \begin{mylist}{(iii)}
  \item
    $T=\{x\}\cup \bigcup_{y\in N_x} T_{x,y}$. 
  \item
    For $y\in N_x$ and $z\in T_{x,y}$, we have
    $[x,z]=\{xy\}\cup [y,z]$. Moreover, the endvertices of an edge in
    $[x,z]$ lie in $\{x\}\cup T_{x,y}$.
  \item
    $A\setminus\{x\}=\bigcup_{y\in N_x} A_{x,y}$ and
    $[x,A]=\bigcup_{y\in N_x} [x,A_{x,y}]$.
  \item
    Given $y\in N_x$, if $A_{x,y}\neq\emptyset$, then
    $[x,A_{x,y}]=\{xy\}\cup [y,A_{x,y}]$.
  \end{mylist}
  Furthermore, all unions above are disjoint unions.
\end{lem}
\begin{proof}
  Given a vertex $z\neq x$, let $p(x,z)=(x_0,x_1,\dots,x_n)$. Then
  $n\geq 1$, $y=x_1$ is a neighbour of $x_0=x$, and
  $z\in T_{x,y}$. It is clear that $x\notin T_{x,y}$ for any
  $y\in N_x$. Moreover, it is immediate from definition that
  $T_{x,y}\cap T_{x,z}=\emptyset$ for distinct neighbours $y,z$ of
  $x$. Thus, (i) follows.

  To see (ii), let $p(x,z)=(x_0,x_1,\dots,x_n)$. Then $n\geq 1$,
  $x_0=x$, $x_1=y$ and $p(y,z)=(x_1,\dots,x_n)$. Hence
  $[x,z]=\{xy\}\cup[y,z]$, and $xy\notin[y,z]$ since the vertices
  $x_0,x_1,\dots,x_n$ are pairwise distinct. For $1\leq i\leq n$, we
  have $p(x,x_i)=(x_0,x_1,\dots,x_i)$, and hence $p_1(x,x_i)=y$. This
  implies that $\{x_0,x_1,\dots,x_n\}\subset \{x\} \cup T_{x,y}$ and
  the second part of~(ii) follows.

  It follows from~(i) that
  $A\setminus\{x\}=\bigcup_{y\in N_x} A_{x,y}$ and that this is a
  disjoint union. The second part of~(iii) now follows:
  \[
    [x,A]=\bigcup_{a\in A} [x,a]=\bigcup_{y\in N_x} \bigcup_{a\in
      A_{x,y}} [x,a]= \bigcup_{y\in N_x} [x,A_{x,y}]\ ,
  \]
  and moreover, since the endvertices of an edge in $[x,A_{x,y}]$ lie
  in $\{x\}\cup T_{x,y}$, it follows that the sets $[x,A_{x,y}]$,
  $y\in N_x$, are pairwise disjoint.

  Finally, we establish (iv). If $A_{x,y}\neq\emptyset$, then
  from~(ii) it follows that
  \[
  [x,A_{x,y}]=\bigcup_{a\in A_{x,y}} [x,a] =\bigcup_{a\in A_{x,y}}
  \{xy\}\cup [y,a] = \{xy\}\cup [y,A_{x,y}]
  \]
  and the union is a disjoint union.
\end{proof}

We next prove Theorem~\ref{thm:tsp-in-tree} in a special case.

\begin{thm}
  \label{thm:tsp-in-tree-special}
  For $x\in T$ and a finite $A\subset T$, we have
  $\tsp_T(x,A,x)=2\bigabs{[x,A]}$.
\end{thm}

\begin{proof}
  We may assume $A\neq \emptyset$ otherwise the result is clear. Let
  us define $h=h(x,A)=\max_{a\in A} d(x,a)$. We construct by recursion
  on~$h$ a walk $w$ from $x$ to $x$ of length $n=2\abs{[x,A]}$
  visiting all vertices in $A$. Together with
  Proposition~\ref{prop:tsp-in-tree-lower}, this will complete the
  proof.

  If $h=0$, then $A=\{x\}$ and the result is clear. Let us now assume
  that $h\geq 1$. Set $N=N_x$ and $A_y=A_{x,y}$ for each $y\in N$. Let
  $M=\{ y\in N:\, A_y\neq \emptyset\}$.

  Fix $y\in M$. For every $a\in A_y$, we have $p(x,a)=(x,p(y,a))$, and
  thus, $h(y,A_y)\leq h-1$. By recursion, there is a walk $w^{(y)}$
  from $y$ to $y$ of length $2\bigabs{[y,A_y]}$ visiting all vertices of
  $A_y$. Now since $A$ is finite, so is $M$, which we can then
  enumerate as $y_1,y_2,\dots,y_k$. Then
  \[
  w=\big(x,w^{(y_1)},x,w^{(y_2)},x,\dots,x,w^{(y_k)},x\big)
  \]
  is a walk from $x$ to $x$ visiting all vertices in
  $\bigcup_{y\in M} A_y$. It follows from Lemma~\ref{lem:subtrees}
  that $w$ visits all vertices in $A$ and has length
  \[
  2k+\sum_{y\in M} 2\bigabs{[y,A_y]} = \sum_{y\in M} 2\bigabs{\{xy\}\cup
    [y,A_y]} =\sum_{y\in M} 2\bigabs{[x,A_y]}=2\bigabs{[x,A]}\ .
  \]
\end{proof}

We are finally ready to complete the proof of our main result.

\begin{proof}[Proof of Theorem~\ref{thm:tsp-in-tree}]
  We proceed by induction on $d_{T}
  (x,y)$ and construct a walk from $x$ to
  $y$ of length $2\bigabs{[x,A]\setminus [x,y]}+\bigabs{[x,y]}$ visiting
  all vertices of $A$. Together with
  Proposition~\ref{prop:tsp-in-tree-lower}, this will complete the
  proof.

  When $d_{T}(x,y)=0$, the result follows from
  Theorem~\ref{thm:tsp-in-tree-special}. Now assume $d_{T}(x,y)\geq 1$ and
  set $x_1=p_1(x,y)$. Let $N=N_x$ and $A_y=A_{x,y}$ for all
  $y\in N$. Set $A_0=\bigcup_{z\in N, z\neq x_1} A_z$ and
  $A_1=A_{x_1}$. From Lemma~\ref{lem:subtrees} we have the following.
  \begin{equation*}
    \begin{aligned}
      A\setminus\{x\} &= \bigcup_{z\in N_x} A_z =A_0\cup A_1\\
      [x,A_1] &= \{xx_1\}\cup [x_1,A_1] \quad \text{if
        $A_1\neq\emptyset$} \\
      [x,y] &= \{xx_1\}\cup [x_1,y]\\
      [x,A] &= \bigcup_{z\in N_x} [x,A_z] = [x,A_0]\cup [x,A_1]\\
      [x,A_0]\cap [x,y] &= \emptyset
    \end{aligned}
  \end{equation*}
  and moreover, all unions are disjoint unions. From this we obtain
  \begin{equation}
    \label{eq:tsp-in-tree}
          [x_1,A_1]\setminus [x_1,y]=[x,A_1]\setminus [x,y]
          \quad\text{and}\quad 
          [x,A]\setminus[x,y]=[x,A_0]\cup
          \big([x,A_1]\setminus[x,y]\big)
  \end{equation}
  By Theorem~\ref{thm:tsp-in-tree-special}, there is a walk $w^{(0)}$
  from $x$ to $x$ of length $\ell_0=2\bigabs{[x,A_0]}$ visiting all
  vertices in $A_0$. By induction hypothesis, there is a walk
  $w^{(1)}$ from $x_1$ to $y$ of length
  $\ell_1=2\bigabs{[x_1,A_1]\setminus [x_1,y]}+\bigabs{[x_1,y]}$ visiting
  all vertices in $A_1$. It follows that $w=(w^{(0)},w^{(1)})$ is a walk
  from $x$ to $y$ visiting all vertices in
  $\{x\}\cup A_0\cup A_1=\{x\}\cup A$. Let $\ell$ be the length of
  $w$. Then from~\eqref{eq:tsp-in-tree} we obtain
  \begin{equation*}
    \begin{split}
      \ell &= \ell_0+\ell_1+1=2\bigabs{[x,A_0]} +
      2\bigabs{[x_1,A_1]\setminus [x_1,y]}+\bigabs{[x_1,y]} + 1 \\
      &= 2\bigabs{[x,A_0]} + 2\bigabs{[x,A_1]\setminus
        [x,y]}+\bigabs{[x,y]} = 2\bigabs{[x,A]\setminus
        [x,y]}+\bigabs{[x,y]}\ ,
    \end{split}
  \end{equation*}
  as required.
\end{proof}

\begin{rem}
  It is not hard to see that the proof of
  Theorem~\ref{thm:tsp-in-tree} yields an efficient algorithm for
  finding optimal walks in $T$ for the Travelling Salesman Problem.
\end{rem}

\subsection{Embeddability into Hamming cubes}

We will show that the lamplighter graph over a tree bi-Lipschitzly embeds
into a Hamming cube, and thus prove Theorem~\ref{mainthm:A}. Let us fix
a tree $T$. By Proposition~\ref{prop:graph-metric-in-lamplighter} and
Theorem~\ref{thm:tsp-in-tree}, the graph metric in the lamplighter
graph $\la(T)$ is given by
\begin{equation}
  \label{eq:metric-in=lamplighter-of-tree}
  d_{\la(T)}\big((A,x),(B,y)\big) = 2\bigabs{[x,A\symdif B]\setminus
    [x,y]}+\bigabs{[x,y]}+\abs{A\symdif B}\ .
\end{equation}
For $C\subset T$ let $[C]=\bigcup _{x,y\in C}[x,y]$ be the minimal set of
edges needed to travel between different vertices of $C$. Define
\[
I=\big\{ (e,C):\, e\in E(T),\ \emptyset\neq C\subset T,\ C\text{
  finite},\ e\notin [C]\big\}\ .
\]
For $A\subset T$, $x\in T$ and $e\in E(T)$, let
$A_{x,e}=\big\{a\in A:\, e\in[x,a]\big\}$. We now define a map into the
Hamming cube $\ham_I$ whose role is to capture the first of the three
summands in the right-hand side
of~\eqref{eq:metric-in=lamplighter-of-tree}.
\begin{lem}
  \label{lem:lamplighter-into-cube-prelim}
  Define $f\colon\la(T)\to \ham_I$ as follows. For
  $(A,x)\in\la(T)$ and $i\in I$ we let
  \[
  f(A,x)_i=1 \iff \E e\in E(T)\ A_{x,e}\neq\emptyset \text{ and }
  i=(e,A_{x,e})\ .
  \]
  Then for vertices $(A,x)$ and $(B,y)$ of $\la(T)$ we have
  \[
  \bigabs{[x,A\symdif B]\setminus [x,y]} \leq d_{\ham}\big(
  f(A,x),f(B,y)\big) \leq 2\bigabs{[x,A\symdif B]\setminus[x,y]} +
  2\bigabs{[x,y]}\ .
  \]
\end{lem}
\begin{proof}
  We first check that $f$ is well-defined, \ie that $f(A,x)$ has
  finite support. Given $i=(e,C)\in I$, if $f(A,x)_i=1$, then
  $C=A_{x,e}\neq\emptyset$, and hence $e\in[x,A]$. It follows that the
  support of $f(A,x)$ has at most (in fact, exactly) $\bigabs{[x,A]}$
  elements. Since $A$ is finite, so is $[x,A]$, and hence $f(A,x)$ is
  finitely supported.

  We now turn to the inequalities. Given $i=(e,C)\in I$, we have
  \[
  f(A,x)_i\neq f(B,y)_i \iff A_{x,e}\neq B_{y,e} \text{ and
    \emph{either }}C=A_{x,e} \text{\ \emph{or }}C=B_{y,e}\ .
  \]
  Thus, setting
  $E=\{e\in E(T):\, A_{x,e}\neq B_{y,e} \text{ and \emph{either
  }}A_{x,e}\neq\emptyset \text{\ \emph{or }}B_{y,e}\neq\emptyset\}$,
  we have
  \begin{equation}
    \label{eq:lip-inequality-1}
    \abs{E}\leq d_\ham\big( f(A,x),f(B,y)\big) \leq 2\abs{E}\ .
  \end{equation}
  To estimate $\abs{E}$, let us first consider an edge
  $e\in E\setminus[x,y]$. By definition of $E$, there is a vertex
  $c\in A_{x,e}\symdif B_{y,e}$. Hence, using
  Lemma~\ref{lem:off-road-edges}, we have
  $e\in [x,c]\setminus [x,y]=[y,c]\setminus [x,y]$. It follows that
  $c\in A\symdif B$, and thus
  $e\in [x,A\symdif B]\setminus[x,y]$. This shows the upper bound
  \begin{equation}
    \label{eq:lip-inequality-2}
    \abs{E}\leq  \bigabs{[x,A\symdif B]\setminus[x,y]} +
    \bigabs{[x,y]}\ .
  \end{equation}
  Next consider $e\in[x,A\symdif B]\setminus [x,y]$. Then, using
  Lemma~\ref{lem:off-road-edges} again, we have some $c\in A\symdif B$
  such that $e\in [x,c]\setminus [x,y]=[y,c]\setminus [x,y]$. It
  follows that $c\in A_{x,e}\symdif B_{y,e}$, and hence $e\in E$. This
  yields the lower bound
  \begin{equation}
    \label{eq:lip-inequality-3}
    \bigabs{[x,A\symdif B]\setminus [x,y]} \leq \abs{E}\ .
  \end{equation}
  Combining the inequalities~\eqref{eq:lip-inequality-1},
  \eqref{eq:lip-inequality-2} and~\eqref{eq:lip-inequality-3}
  completes the proof of the lemma.
\end{proof}

After some definitions, we will state and prove the main result of
this section, which then immediately yields Theorem~\ref{mainthm:A}.
Note that for disjoint sets $J$ and $K$, the product
$\ham_J\cprod \ham_K$ is the Hamming cube $\ham_{J\cup K}$. In the
next result we identify the vertices of a Hamming cube $\ham_J$ with
finite subsets of $J$.

\begin{thm}
  \label{thm:lamplighter-into-cube}
  Let $T$ be a (non-empty) tree. Let $f\colon \la(T)\to \ham_I$ be
  the map from Lemma~\ref{lem:lamplighter-into-cube-prelim}. Fix
  $x_0\in T$ and define
  $F\colon \la(T)\to \ham_I\cprod \ham_{E(T)}\cprod \ham_T$ by
  \[
  F(A,x)=\big( f(A,x),[x_0,x],A \big)\ .
  \]
  Then $F$ is a bi-Lipschitz embedding with distortion at most~6.
\end{thm}
\begin{proof} 
Fix two vertices $(A,x)$ and $(B,y)$ in $\la(T)$. Then
  \[
  d_{\square}\big( F(A,x),F(B,y)\big) = d_{\ham}\big( f(A,x),f(B,y)\big) +
  \bigabs{[x_0,x]\symdif[x_0,y]} + \bigabs{A\symdif B}\ .
  \]
  We first estimate the middle term. Let $p(x_0,x)=(x_0,x_1,\dots,x_m)$ and
  $p(x_0,y)=(y_0,y_1,\dots,y_n)$. As in the proof of
  Lemma~\ref{lem:off-road-edges}, if $i\leq\min(m,n)$ is maximal such
  that $x_j=y_j$ for $0\leq j\leq i$, then
  \[
  p(y,x)=(y_n,y_{n-1},\dots,y_{i+1},y_i=x_i,x_{i+1},\dots,x_m)\ .
  \]
  It follows at once that
  \[
  [x_0,x] \symdif [x_0,y] = [x,y]\ .
  \]
  Hence, using~\eqref{eq:metric-in=lamplighter-of-tree} and
  Lemma~\ref{lem:lamplighter-into-cube-prelim}, we deduce that
  \begin{align*}
    d_{\la(T)}\big( (A,x),(B,y)\big) &= 2\bigabs{[x,A\symdif B]\setminus[x,y]}
    + \bigabs{[x,y]} + \bigabs{A\symdif B}\\
    &\leq 2\cdot d_{\ham}\big( f(A,x),f(B,y)\big) + \bigabs{[x_0,x] \symdif
      [x_0,y]} + \bigabs{A\symdif B}\\
    &\leq 2\cdot d_{\square}\big( F(A,x),F(B,y) \big)
  \end{align*}
  and that
  \begin{align*}
    d_{\square}\big( F(A,x),F(B,y) \big) &\leq 
    2\bigabs{[x,A\symdif B]\setminus[x,y]} + 3\bigabs{[x,y]} +
    \bigabs{A\symdif B}\\
    &\leq 3\cdot d_{\la(T)}\big( (A,x),(B,y)\big)\ .
  \end{align*}
\end{proof}

\section{Lamplighter graph over the vertex-coalescence of two graphs}

\label{sec:amalgamation}

The procedure that consists of gluing two graphs at a common vertex
is known as \emph{vertex-coalescence }or, simply, \emph{coalescence
}of two graphs. Consider two \emph{pointed graphs }$G_1=(V_1,E_1,v_1)$
and $G_2=(V_2,E_2,v_2)$, \ie graphs $G_i$, $i=1,2$, with vertex set
$V_i$, edge set $E_i$ and a specified vertex $v_i\in V_i$.
We define \emph{the vertex-coalescence $G_1\coa G_2$ of $G_1$ and
  $G_2$ }by first taking the disjoint union of $G_1$ and $G_2$
followed by identifying the vertices $v_1$ and $v_2$. Formally,
$G_1\coa G_2$ has vertex set
\[
V=\big\{ (x,i):\, x\in V_i\setminus\{v_i\},\ i=1,2\big\} \cup \{v_0\}
\]
where $v_0\notin V_1\times\{1\}\cup V_2\times\{2\}$, and edge set
\begin{align*}
  E=\big\{ \big((x,i),(y,i)\big)&:\, x,y\in
  V_i\setminus\{v_i\},\ xy\in E_i,\ i=1,2\big\}\\
  &\cup \big\{ \big((x,i),v_0\big):\, x\in
  V_i\setminus\{v_i\},\ xv_i\in E_i,\ i=1,2\big\}\ .
\end{align*}
This formal definition is rather cumbersome. In practice, we shall
either assume after relabeling that $V_1\cap V_2=\{v_0\}$ and
$v_0=v_1=v_2$, in which case we can simply take $V=V_1\cup V_2$ and
$E=E_1\cup E_2$, or, particularly in the case of gluing several copies
of the same pointed graph together, we shall refer to vertices of the
original graph as being in the $k^{\text{th}}$ copy in the coalesced
graph for $k=1,2,\dots$. Note that if $G_1$ and $G_2$ are connected,
then so is $G_1\coa G_2$ with graph metric given by
\begin{equation}
  \label{eq:graph-metric-of-coalescence}
  d_{G_1\coa G_2}(u,v)=%
  \begin{cases}
    d_{G_i}(u,v) & \text{if }u,v\in G_i,\\
    d_{G_1}(u,v_0)+d_{G_2}(v_0,v) & \text{if } u\in G_1, v\in G_2\ .
  \end{cases}
\end{equation}

In the next lemma we record the relationship between the Travelling
Salesman Problem on the coalescence graph with the ones on its
components. The proof is elementary and left to the reader. 
	
\begin{lem}
  Let $G_1\coa G_2$ be the coalescence of two connected pointed graphs
  $G_1$ and $G_2$ at a common vertex $v_0$. Let
  $x,y\in G_1\coa G_2$ and $C\subset G_1\coa G_2$ with $C$ finite.

  \vspace{2ex}

  \noindent
  If there exists $i\in\{1,2\}$ such that $x,y\in G_i$ and
  $C\subset G_i$, then
  \begin{equation}
    \label{eq:15}
    \tsp_{G_1\coa G_2}(x,C,y)=\tsp_{G_i}(x,C,y)\ .
  \end{equation}

  \vspace{2ex}

  \noindent
  If $x\in G_1$, $y\in G_2$ and $C=C_1\cup C_2$ with $C_i\subset G_i$
  for $i=1,2$, then
  \begin{equation}
    \label{eq:16}
    \tsp_{G_1\coa G_2}(x,C,y) = \tsp_{G_1}(x,C_1,v_0) +
    \tsp_{G_2}(v_0,C_2,y)\ .
  \end{equation}

  \vspace{2ex}

  \noindent
  If $x\in G_1$, $y\in G_1$ and $C\cap G_2\neq \emptyset$, then
  \begin{multline}
    \label{eq:17}
    \tsp_{G_1\coa G_2}(x,C,y)\\
    = \min \big\{\tsp_{G_1}(x,C',v_0) +
    \tsp_{G_2}(v_0,C\cap V_2,v_0) + \tsp_{G_1}(v_0,C'',y)\big\}
  \end{multline}
  where the minimum is taken over all  sets
  $C', C''\subset V_1$ with $C'\cup C''=C\cap V_1$.
\end{lem}					 		

The purpose of the next theorem is to establish a metric connection
between the lamplighter graph over the coalescence of two graphs with
the lamplighter graphs over its components. To do this, we need to
make use of clover graphs. Given $n\in \bn$ and a pointed graph
$G=(V,E,v_0)$, \emph{the clover graph }$\clo(G,n)$ is obtained by
coalescing $n$ copies of $G$ at $v_0$ in an obvious inductive
fashion.

\begin{thm}
  \label{thm:coalescence-lemma}
  Let $G_1\coa G_2$ be the coalescence of two finite, connected
  pointed graphs $G_1$ and $G_2$ at a common vertex $v_0$. Then there
  exists a map
  \[
  f\colon\la(G_1\coa G_2)\to\la(G_1)\cprod\la(G_2)\cprod%
  \clo\big(G_1,2^{\abs{G_2}}\big)\cprod\clo%
  \big(G_2,2^{\abs{G_1}}\big)
  \]
  such that
  \begin{equation}
    \label{eq:coalescence-lemma}
    d_{\la(G_1\coa G_2)}(u,v)\leq d_{\square}\big(f(u),f(v)\big)\leq
    2\cdot d_{\la (G_1\coa G_2)}(u,v)
  \end{equation}
  for all $u,v\in\la(G_1\coa G_2)$.
\end{thm}		
\begin{proof}
  Observe that $2^{\abs{G_2}}$ is the number of subsets of $G_2$, and
  thus we can index the $2^{\abs{G_2}}$ copies of $G_1$ in
  $\clo\big(G_1,2^{\abs{G_2}}\big)$ by the collection of all subsets of
  $G_2$. For $x\in G_1$ and $S\subset G_2$ we denote by $\iota_S(x)$
  the element $x$ considered in the copy of $G_1$ in
  $\clo\big(G_1,2^{\abs{G_2}}\big)$ that is indexed by $S$. We proceed
  in a similar way for $\clo(G_2,2^{\abs{G_1}})$. We define the
  function
  \[
  f\colon\la(G_1\coa G_2)\to\la(G_1)\cprod\la(G_2)\cprod%
  \clo\big(G_1,2^{\abs{G_2}}\big)\cprod\clo%
  \big(G_2,2^{\abs{G_1}}\big)
  \]
  as follows. Given a vertex $(A,x)$ of $\la(G_1\coa G_2)$, we let
  $A_i=A\cap G_i$ for $i=1,2$, and set
  \begin{equation}
    f(A,x)=%
    \begin{cases}
      \big((A_1,x), (A_2,v_0), \iota_{A_2}(x), v_0\big) & \text{if }
      x\in G_1,\\
      \big((A_1,v_0), (A_2,x), v_0, \iota_{A_1}(x)\big) & \text{if }
      x\in G_2\ .
    \end{cases}
  \end{equation}
  To establish~\eqref{eq:coalescence-lemma}, we fix vertices $(A,x)$
  and $(B,y)$ in $\la (G_1\coa G_2)$, we let $A_i=A\cap G_i$ and
  $B_i=B\cap G_i$ for $i=1,2$, and consider several cases. To make the
  notation less crowded, we will at times drop subscripts in the graph
  distance.
  \begin{mylist}{}
  \item[\underline{Case 1:}] $x,y\in G_1$ and $A\symdif B\subset G_1$.
    In this case we have $A\symdif B=A_1\symdif B_1$ and $A_2=B_2$. It
    follows that
    \begin{align*}
      &d_{\square}\big(f(A,x),f(B,y)\big) & \\
      &= d\big((A_1,x),(B_1,y)\big) +
      d\big((A_2,v_0),(B_2,v_0)\big)+
      d\big(\iota_{A_2}(x),\iota_{B_2}(y)\big)+d(v_0,v_0)\\
      &=\tsp_{G_1}(x,A_1\symdif B_1,y)+
      \abs{A_1\symdif B_1}+ 
      \tsp_{G_2}(v_0,A_2\symdif B_2,v_0)+\abs{A_2\symdif B_2}\\
      & \qquad + d\big(\iota_{A_2}(x),\iota_{B_2}(y)\big) &
      \text{\makebox[0pt][r]{(by~\eqref{eq:graph-metric-in-lamplighter})}}\\
      &=\tsp_{G_1\coa G_2}(x,A\symdif B,y)+\abs{A\symdif B}+d_{G_1}(x,y)
      &
      \text{\makebox[0pt][r]{(by~\eqref{eq:15}
          and~\eqref{eq:graph-metric-of-coalescence})}}\\
      &=d_{\la (G_1\coa G_2)}\big((A,x),(B,y)\big)+d_{G_1\coa G_2}(x,y) &
      \text{\makebox[0pt][r]{(by~\eqref{eq:graph-metric-in-lamplighter}
          and~\eqref{eq:graph-metric-of-coalescence})}}
    \end{align*}
    which implies~\eqref{eq:coalescence-lemma} since
    \[
    0\leq d_{G_1\coa G_2}(x,y)\leq \tsp_{G_1\coa G_2} (x,A\symdif B,y)
    \leq d_{\la (G_1\coa G_2)}\big((A,x),(B,y)\big)\ .
    \]
  \item[\underline{Case 2:}]
    $x\in G_1$ and $y\in G_2$. Then,
    \begin{align*}
      &d_\square\big(f(A,x),f(B,y)\big) &\\
      &=d\big((A_1,x),(B_1,v_0)\big)+d\big((A_2,v_0),(B_2,y)\big)+
      d\big(\iota_{A_2}(x),v_0\big)+d\big(v_0,\iota_{B_1}(y)\big)\\
      &=\tsp_{G_1}(x,A_1\symdif B_1,v_0)+\abs{A_1\symdif B_1}+
      \tsp_{G_2}(v_0,A_2\symdif B_2,y)+\abs{A_2\symdif B_2}\\ 
      & \qquad +d_{G_1}(x,v_0)+d_{G_2}(v_0,y) &
      \text{\makebox[0pt][r]{(by~\eqref{eq:graph-metric-in-lamplighter}
          and~\eqref{eq:graph-metric-of-coalescence})}}\\
      &=\tsp_{G_1\coa G_2}(x,A\symdif B,y)+\abs{A\symdif B}+
      d_{G_1}(x,v_0)+d_{G_2}(v_0,y) &
      \text{\makebox[0pt][r]{(by~\eqref{eq:16})}}\\
      &=d_{\la (G_1\coa G_2)}\big((A,x),(B,y)\big)+
      d_{G_1\coa G_2}(x,y) &
      \text{\makebox[0pt][r]{(by~\eqref{eq:graph-metric-in-lamplighter}
          and~\eqref{eq:graph-metric-of-coalescence})}}
    \end{align*}
    and we are again done as in the previous case.
  \item[\underline{Case 3:}]
    $x,y\in G_1$ and $A\symdif B\cap G_2\neq\emptyset$. Then
    $A_2\symdif B_2\neq\emptyset$, and thus $A_2\neq B_2$. Therefore,
    \begin{align*}
      &d_{\square}\big(f(A,x),f(B,y)\big) & \\
      &=d\big((A_1,x),(B_1,y)\big)+d\big((A_2,v_0),(B_2,v_0)\big)+
      d\big(\iota_{A_2}(x),\iota_{B_2}(y)\big)+d(v_0,v_0)\\
      &=\tsp_{G_1}(x,A_1\symdif B_1,y)+\abs{A_1\symdif B_1}+
      \tsp_{G_2}(v_0,A_2\symdif B_2,v_0)+\abs{A_2\symdif B_2}\\
      & \qquad +d_{G_1}(x,v_0)+d_{G_1}(v_0,y) &
      \text{\makebox[0pt][r]{(by~\eqref{eq:graph-metric-in-lamplighter}
          and~\eqref{eq:graph-metric-of-coalescence})}}\\
      &=\tsp_{G_1}(x,A_1\symdif B_1,y)+
      \tsp_{G_2}(v_0,A_2\symdif B_2,v_0) + d_{G_1}(x,v_0) +
      d_{G_1}(v_0,y)\\
      & \qquad + \abs{A\symdif B}\ .
    \end{align*}
    For any decomposition $C\cup D$ of $A_1\symdif B_1$ one has 
    \begin{align*}
      \tsp_{G_1}(x,A_1\symdif B_1,y) &\leq \tsp_{G_1}(x,C,v_0)+
      \tsp_{G_1}(v_0,D,y)\\
      d_{G_1}(x,v_0) &\leq \tsp_{G_1}(x,C,v_0)\\
      d_{G_1}(v_0,y) &\leq \tsp_{G_1}(v_0,D,y)\ ,
    \end{align*}
    and hence
    \begin{align*}
      &d_{\square}\big(f(A,x),f(B,y)\big) &\\
      &\leq 2\cdot \tsp_{G_1}(x,C,v_0)+ 
      \tsp_{G_2}(v_0,A_2\symdif B_2,v_0)+ 2\cdot\tsp_{G_1}(v_0,D,y)
      +\abs{A\symdif B}\\
      &\leq 2\big(\tsp_{G_1}(x,C,v_0)+
      \tsp_{G_2}(v_0,A_2\symdif B_2,v_0)+\tsp_{G_1}(v_0,D,y)\big)
      +\abs{A\symdif B}\\
      &= 2\cdot\tsp_{G_1\coa G_2}(x,A\symdif B,y)+
      \abs{A\symdif B} &
      \text{\makebox[0pt][r]{(for some choice of $C,D$
          by~\eqref{eq:17})}}\\
      &\leq 2\cdot d_{\la (G_1\coa G_2)}\big((A,x),(B,y)\big) &
      \text{\makebox[0pt][r]{(by~\eqref{eq:graph-metric-in-lamplighter})}}
    \end{align*}
    For the lower bound, assume without loss of generality that
    $d_{G_1}(x,v_0)\geq d_{G_1}(v_0,y)$. Since
    $\tsp_{G_1}(x,A_1\symdif B_1,y)+d_{G_1}(v_0,y)\geq
    \tsp_{G_1}(x,A_1\symdif B_1,v_0)$, it follows that
    \begin{align*}
      &d_{\square}\big(f(A,x),f(B,y)\big) &\\
      &\geq \tsp_{G_1}(x,A_1\symdif B_1,v_0)+
      \tsp_{G_2}(v_0,A_2\symdif B_2,v_0)+d_{G_1}(x,v_0)+
      \abs{A\symdif B}\\
      &\geq\tsp_{G_1}(x,A_1\symdif B_1,v_0)+
      \tsp_{G_2}(v_0,A_2\symdif B_2,v_0)+
      \tsp_{G_1}(v_0,\emptyset,y)\\
      & \qquad + \abs{A\symdif B}\\
      &\geq \tsp_{G_1\coa G_2} (x,A\symdif B,y) + \abs{A\symdif B} 
      & \text{\makebox[0pt][r]{(by~\eqref{eq:17})}}\\
      &= d_{\la (G_1\coa G_2)}\big((A,x),(B,y)\big)
      & \text{\makebox[0pt][r]{(by~\eqref{eq:graph-metric-in-lamplighter})}}
    \end{align*}
  \end{mylist}
\end{proof}
We will now illustrate the utility of
Theorem~\ref{thm:coalescence-lemma} with two applications. The first
result (Proposition~\ref{prop:star-characterization}) is concerned with
embeddings of lamplighter graphs over star graphs into
non-superreflexive Banach spaces. Given $k,n\in\bn$, we define the
\emph{star graph $\st_{n,k}$ }to be the clover graph $\clo(\pa_k,n)$
obtained by coalescing $n$ copies of a path of length~$k$ at an
endvertex (see Figure~\ref{fig:star-rose}).
\begin{lem}
  \label{lem:star-lemma}
  Let $k,n\in\bn$ and let $(E,\norm{\cdot})$ be an $n$-dimensional
  Banach space. Then $\st_{n,k}$ bi-Lipschitzly embeds into~$E$ with
  distortion at most~$2$.
\end{lem}

\begin{proof}
  Let $(e_i)_{i=1}^n$ be an Auerbach basis for $E$. By this we mean
  that $\norm{e_i}=1$ for $i=1,\ldots, n$ and that there are
  functionals $(e^*_i)_{i=1}^n$ in the dual of $E$ which are also
  normalized and such that $e^*_i(e_j)=\delta_{ij}$ for all
  $i,j=1,\ldots,n$. Define $f\colon \st_{n,k}\to E$ by
  $f(x)=d_{\pa_k}(v_0,x)e_i$ if $x$ belongs to the $i^{\text{th}}$
  copy of $\pa_k$ in $\st_{n,k}$. Here $v_0$ is the endvertex of $\pa_k$
  at which the $n$ copies of $\pa_k$ are coalesced. If, for some $i$,
  both $x$ and $y$ belong to the $i^{\text{th}}$ copy of $\pa_k$ in
  $\st_{n,k}$, then
  \begin{align*}
    \norm{f(x)-f(y)} &= \abs{d_{\pa_k}(v_0,x)-d_{\pa_k}(v_0,y)} \cdot
    \norm{e_i}\\
    &=d_{\pa_k}(x,y)=d_{\st_{n,k}}(x,y)\ .
  \end{align*}
  If, for some $i\neq j$, we have that $x$ belongs to the
  $i^{\text{th}}$ copy and $y$ to the $j^{\text{th}}$ copy of $\pa_k$ in
  $\st_{n,k}$, then
  \begin{align*}
    \norm{f(x)-f(y)}&=\norm{d_{\pa_k}(v_0,x)e_i-d_{\pa_k}(v_0,y)e_j}\\
    &\geq \max\{d_{\pa_k}(v_0,x),d_{\pa_k}(v_0,y)\}\\
    &\geq \tfrac12 \big(d_{\pa_k}(v_0,x)+d_{\pa_k}(v_0,y)\big)\\
    &= \tfrac12\cdot d_{\st_{n,k}}(x,y)\ .
  \end{align*}
  On the other hand, $f$ is clearly $1$-Lipschitz by the triangle
  inequality.
\end{proof}

The following lemma says that under certain conditions one can embed a
finite product of metric spaces into a Banach space if the metric
spaces are themselves embeddable in a particular fashion. Similar
arguments have already been used in metric geometry
(\cf~\cite{ostrovskii:14}*{Theorem~1.7}) and their proofs simply rely
on basic functional analytic principles. We provide a proof for the
convenience of the reader unfamiliar with those.

\begin{lem}
  \label{lem:product-lemma}
  Let $M_1,\dots,M_n$ be metric spaces, and let $Y$ be
  an infinite-dimensional Banach space. Assume that there exist
  positive real numbers $D_1,\dots, D_n$ such that for every
  $i=1,\dots,n$ and for every finite-codimensional subspace $Z$ of
  $Y$, there is a bi-Lipschitz embedding $\varf_{i,Z}$ of $M_i$ of
  distortion at most $D_i$ into a finite-dimensional subspace of
  $Z$. Then for every $\vare>0$, the product $M=\prod_{i=1}^nM_i$
  equipped with the $\ell_1$-metric bi-Lipschitzly embeds into $Y$
  with distortion at most $(2+\vare)n\max_{1\leq i\leq n}D_i$.
\end{lem}

\begin{proof}
  We begin with a basic result from the geometry of Banach
  spaces. Given $\delta>0$ and a finite-dimensional subspace $E$ of
  $Y$, there is a finite-codimensional subspace $Z$ of $Y$ such that
  $\norm{x+z}\geq (1-\delta)\norm{x}$ for all $x\in E$ and
  $z\in Z$. Indeed, choose a $\delta$-net $x_1,\dots,x_K$ in the unit
  sphere of $E$ together with norming functionals $x_1^*,\dots,x_K^*$
  in $Y^*$. Set $Z=\bigcap _{i=1}^K \ker x^*_i$. Given $x\in E$ and
  $z\in Z$, assuming as we may that $\norm{x}=1$, choose
  $i\in\{1,\dots,K\}$ such that $\norm{x-x_i}\leq\delta$. Then we have
  \[
  \norm{x+z}\geq \norm{x_i+z}-\delta\geq
  \abs{x^*_i(x_i+z)}-\delta=1-\delta\ ,
  \]
  as required.
  
  Let us now turn to the statement of the lemma. Firstly, after
  scaling, we may assume that for every $i=1,\dots,n$ and for every
  finite-codimensional subspace $Z$ of $Y$ we have
  \[
  d_{M_i}(u,v)\leq \norm{\varf_{i,Z}(u)-\varf_{i,Z}(v)}
  \leq D_i\cdot d_{M_i}(u,v) \qquad \text{for all }u,v\in M_i\ .
  \]
  Fix $\vare>0$ and choose $\delta>0$ satisfying
  $2(1-\delta)^{-n}<2+\vare$. We will recursively construct
  finite-codimensional subspaces $Z_1,\dots,Z_n$ of $Y$ together with
  finite-dimensional subspaces $E_i$ of $Z_i$ as follows. At the
  $j^{\text{th}}$ step, having chosen $E_i\subset Z_i$ for
  $1\leq i<j$, we choose a finite-codimensional subspace $Z_j$ of $Y$
  such that $\norm{x+z}\geq(1-\delta)\norm{x}$ for all
  $x\in E_1+\dots+E_{j-1}$ and for all $z\in Z_j$. We then choose a
  finite-dimensinal subspace $E_j$ of $Z_j$ containing
  $\varf_{j,Z_j}(M_j)$. This completes the recursive construction,
  which has the following consequence. Given $x_i\in E_i$ for
  $i=1,\dots,n$, we have
  \[
  (1-\delta)^{n-m}\cdot \Bignorm{\sum_{i=1}^m x_i}\leq
  \Bignorm{\sum_{i=1}^n x_i}\ .
  \]
  for each $m=1,\dots,n$. It follows by the triangle-inequality and
  by the choice of $\delta$ that
  \begin{equation}
    \label{eq:pseudo-orthogonal}
    \max _{1\leq m\leq n} \norm{x_m} \leq (2+\vare)
    \Bignorm{\sum_{i=1}^n x_i}\ .
  \end{equation}
  We now define $\varf\colon M\to Y$ by
  $\varf(\vu)=\sum_{i=1}^n \varf_{i,Z_i} (u_i)$ for
  $\vu=(u_1,\dots,u_n)$ in the product space $M=\prod_{i=1}^nM_i$. We
  claim that $\varf$ is bi-Lipschitz with distortion at most
  $(2+\vare)Dn$ where $D=\max_{1\leq i\leq n}D_i$. Let us fix
  $\vu=(u_i)_{i=1}^n$ and $\vv=(v_i)_{i=1}^n$ in $M$. On the one hand,
  the triangle-inequality yields
  \[
  \norm{\varf(\vu)-\varf(\vv)} \leq \sum_{i=1}^n
  \norm{\varf_{i,Z_i} (u_i)-\varf_{i,Z_i} (v_i)} \leq
  \sum_{i=1}^n D_i\cdot d_{M_i}(u_i,v_i) \leq D\cdot d_M(\vu,\vv)\ .
  \]
  On the other hand, using~\eqref{eq:pseudo-orthogonal} we obtain the
  following lower bound.
  \begin{align*}
    (2+\vare)\cdot \norm{\varf(\vu)-\varf(\vv)} &\geq \max_{1\leq i\leq n}
  \norm{\varf_{i,Z_i} (u_i)-\varf_{i,Z_i} (v_i)}\\
  &\geq \max_{1\leq i\leq n} d_{M_i}(u_i,v_i) \geq \tfrac{1}{n}\cdot
  d_M(\vu,\vv)\ .
  \end{align*}
  Thus, $\varf$ has distortion at most $(2+\vare)Dn$, as claimed.
\end{proof}
Using Theorem~\ref{thm:coalescence-lemma} we can show that for fixed
$n\in\bn$, the sequence $(\la(\st_{n,k}))_{k\in\bn}$ of lamplighter
graphs equi-bi-Lipschiztly embeds into any non-superreflexive Banach
space.
\begin{prop}
  \label{prop:star-characterization}
  Let $Y$ be a non-superreflexive Banach space. For all $n\in\bn$,
  there exist $C(n)\in(0,\infty)$ and maps
  $f_{n,k}\colon \la(\st_{n,k})\to Y$ such that
  \begin{equation*}
    d_{\la(\st_{n,k})}(x,y)\leq \norm{f_{n,k}(x)-f_{n,k}(y)}_{Y}\leq
    C(n)\cdot d_{\la(\st_{n,k})}(x,y)
  \end{equation*}
  for all $k\in\bn$ and for all $x,y\in\la(\st_{n,k})$.
\end{prop}
\begin{proof}
  It is sufficient to prove the proposition for each
  $n\in\{2^i:\, i\in\bn\}$. Observe that
  $\st_{2^{i},k}=\st_{2^{i-1},k}\coa \st_{2^{i-1},k}$ and
  $\abs{\st_{2^{i-1},k}}=k\cdot 2^{i-1}+1$. Set
  $\alpha_{i,k}=2^{k\cdot 2^{i-1}+1}$. Applying
  Theorem~\ref{thm:coalescence-lemma}, $\la(\st_{2^{i},k})$
  bi-Lipschitzly embeds with distortion at most~$2$ into
  \[
  \la(\st_{2^{i-1},k})\cprod\la(\st_{2^{i-1},k})\cprod
  \clo(\st_{2^{i-1},k},\alpha_{i,k}) \cprod
  \clo(\st_{2^{i-1},k},\alpha_{i,k})\ .
  \]
  Now observe that $\clo(\st_{r,s},t)=\st_{rt,s}$ for any
  $r,s,t\in\bn$. If we apply Theorem~\ref{thm:coalescence-lemma}
  another $i-1$ times, we obtain that $\la(\st_{2^{i},k})$
  bi-Lipschitzly embeds with distortion at most~$2^i$ into the
  Cartesian product of 
  $4\cdot2^{i-1}+2\cdot2^{i-2}+\dots+2\cdot2+2=3\cdot2^i-2$ graphs
  each of which is either $\la(\pa_k)$ or a graph of the form
  $\st_{r,k}$ for some $r\in\bn$. Note that all these graphs admit
  bi-Lipschitz embeddings into every finite-codimensional subspace of
  any non-superreflexive Banach space. Indeed, by Bourgain's metric
  characterization of superreflexivity~\cite{bourgain:86}, for every
  $\vare>0$, every binary tree of finite height admits a
  bi-Lipschitz embedding into every finite-codimensional subspace of
  any non-superreflexive Banach space with distortion at most
  $1+\vare$. Therefore the conclusion follows by combining this result
  with Lemma~\ref{lem:star-lemma},
  Proposition~\ref{prop:path-into-trees}, and
  Lemma~\ref{lem:product-lemma}.  
\end{proof}

Let $\cyc_k$ denote the $k$-cycle, \ie the cycle of length~$k$ with
vertices $v_1,\dots,v_k$ and edges $v_{i-1}v_i$ for $i=1,\dots,k$,
where we set $v_0=v_k$. Given $k,n\in\bn$, we define the \emph{rose
  graph $\rose_{n,k}$ }to be the clover graph $\clo(\cyc_k,n)$
obtained by coalescing $n$ copies of $\cyc_k$ at $v_0$. Using
Theorem~\ref{thm:coalescence-lemma} together with the main result
from~\cite{RO2018}, we can show that for fixed $n\in\bn$ the sequence
$\big(\la(\rose_{n,k})\big)_{k\in\bn}$ of lamplighter graphs
equi-bi-Lipschiztly embeds into any non-superreflexive Banach
space. First we need to prove that $\rose_{n,k}$ can be well embedded
into Euclidean spaces.

\begin{lem}
  \label{lem:clover-cycles}
  Let $n\in\bn$. There exist maps
  $g_{n,k}\colon \rose_{n,k}\to \ell_2^{2n}$ such that
  \begin{equation*}
    \tfrac{1}{\sqrt{2}}\cdot d_{\rose_{n,k}}(x,y)\leq
    \norm{g_{n,k}(x)-g_{n,k}(y)}\leq \tfrac{\pi}{2}\cdot d_{\rose_{n,k}}(x,y)
  \end{equation*}
  for all $k\in\bn$ and for all $x,y\in\rose_{n,k}$.
\end{lem}

\begin{proof}
  It was proven in~\cite{linial-magen:00} that the natural embedding
  of the $k$-cycle onto the vertices of the regular $k$-gon in $\br^2$
  is optimal and has distortion exactly
  $\frac{k}{2}\sin\big(\frac{\pi}{k}\big) \leq
  \frac{\pi}{2}$. Therefore, there exist 
  maps $\varf_k\colon \cyc_k\to \ell_2^2$ with $\varf_k(v_0)=0$
  and such that
  \[
  d_{\cyc_k}(x,y)\leq \norm{\varf_k(x)-\varf_k(y)}_2\leq
  \tfrac{\pi}{2}\cdot d_{\cyc_k}(x,y)\ .
  \]
  Let $E_i=\ell_2^2$ for all $i\in \bn$ and define 
  \begin{eqnarray*}
    g_{n,k}\colon \rose_{n,k} &\to& (E_1\oplus\dots\oplus E_{i-1}\oplus
    E_i\oplus E_{i+1}\oplus\dots \oplus E_n)_{\ell_2}=\ell_2^{2n}\\
    x &\mapsto& (0,\dots,0,\varf_k(x),0,\dots,0) \qquad \text{if $x$
      belongs to the $i^{\text{th}}$ copy of $\cyc_k$.}
  \end{eqnarray*}
  Observe that if $x$ and $y$ belong to the same copy of $\cyc_k$ in
  $\rose_{n,k}$, then one has
  \[
  \norm{g_{n,k}(x)-g_{n,k}(y)}_2=\norm{\varf_k(x)-\varf_k(y)}_2\ .
  \]
  Otherwise, 
  \begin{align*}
    \norm{g_{n,k}(x)-g_{n,k}(y)}_2 &=
    \sqrt{\norm{\varf_k(x)}_2^2+\norm{\varf_k(y)}_2^2} \leq
    \norm{\varf_k(x)}_2+\norm{\varf_k(y)}_2\\ 
    &\leq
    \tfrac{\pi}{2}\cdot d_{\cyc_k}(x,v_0)+\tfrac{\pi}{2}\cdot d_{\cyc_k}(y,v_0)
    =\tfrac{\pi}{2}\cdot d_{\rose_{n,k}}(x,y)\ ,
  \end{align*}
  and
  \begin{align*}
    \norm{g_{n,k}(x)-g_{n,k}(y)}_2&\geq
    \tfrac{1}{\sqrt{2}}\big(\norm{\varf_k(x)}_2+\norm{\varf_k(y)}_2\big)\\
    &\geq \tfrac{1}{\sqrt{2}}\big(d_{\cyc_k}(x,v_0)+d_{\cyc_k}(y,v_0)\big)
    =\tfrac{1}{\sqrt{2}}\cdot d_{\rose_{n,k}}(x,y)\ .
  \end{align*}
\end{proof}

\begin{prop}
  \label{prop:rose-characterization}
  Let $Y$ be a non-superreflexive Banach space. For all $n\in\bn$
  there exist $D(n)\in(0,\infty)$ and maps
  $f_{n,k}\colon\la(\rose_{n,k})\to Y$ such that
  \begin{equation*}
    d_{\la(\rose_{n,k})}(x,y)\leq \norm{f_{n,k}(x)-f_{n,k}(y)}_{Y}\leq
    D(n)\cdot d_{\la(\rose_{n,k})}(x,y)
  \end{equation*}
  for all $k\in\bn$ and for all $x,y\in\la(\rose_{n,k})$.
\end{prop}

\begin{proof}
  It is sufficient to prove the proposition for each
  $n\in\{2^{i}:\,i\in\bn\}$. Observe that
  $\rose_{2^{i},k}=\rose_{2^{i-1},k}\coa \rose_{2^{i-1},k}$ and
  $\abs{\rose_{2^{i-1},k}}=(k-1)2^{i-1}+1$. Set
  $\beta_{i,k}=2^{(k-1)2^{i-1}+1}$. Applying
  Theorem~\ref{thm:coalescence-lemma}, $\la(\rose_{2^{i},k})$
  bi-Lipschitzly embeds with distortion at most~$2$ into
  \[
  \la(\rose_{2^{i-1},k})\cprod \la(\rose_{2^{i-1},k})\cprod
  \clo(\rose_{2^{i-1},k},\beta_{i,k})\cprod\clo(\rose_{2^{i-1},k},\beta_{i,k})\ .
  \]
  Now observe that $\clo(\rose_{r,s},t)=\rose_{rt,s}$ for any
  $r,s,t\in\bn$. If we apply Theorem~\ref{thm:coalescence-lemma}
  another $i-1$ times, we obtain that $\la(\rose_{2^{i},k})$
  bi-Lipschitzly embeds with distortion at most~$2^i$ into the
  Cartesian product of
  $4\cdot2^{i-1}+2\cdot2^{i-2}+\dots+2\cdot2+2=3\cdot2^i-2$ graphs
  each of which is either $\la(\cyc_k)$ or a graph of the form $\rose_{r,k}$
  for some $r\in\bn$. Note that all these graphs admit bi-Lipschitz
  embeddings into every finite-codimensional subspace of any
  non-superreflexive Banach space. Indeed, it was proved
  in~\cite{RO2018} that $\la(\cyc_k)$ can be embedded into a product
  of 8 trees, and hence one can again use Bourgain's metric
  characterization of superreflexivity~\cite{bourgain:86}. The
  conclusion follows by appealing to Lemma~\ref{lem:clover-cycles},
  Dvoretzky's theorem, and Lemma~\ref{lem:product-lemma}.  
\end{proof}

\begin{rem}
  By carefully keeping track of the distortions of embeddings
  in the proofs of Propositions~\ref{prop:star-characterization}
  and~\ref{prop:rose-characterization}, one obtains order $n^2$
  upper bounds on $C(n)$ and $D(n)$.
\end{rem}

\begin{rem}
  At this point we have established the implications
  ``(ii)$\implies$(i)'' and ``(iii)$\implies$(i)'' in
  Theorem~\ref{mainthm:B}. The remaining implications
  will be shown in Section~\ref{sec:5}.
\end{rem}

\section{Induced maps between lamplighter graphs}

A map $f\colon G\to H$ between two graphs induces a map
$\fb\colon \la(G)\to\la(H)$ defined by $\fb(A,x)=\big(f(A),f(x)\big)$,
where $f(A)=\{f(y):\,y\in A\}$. Moreover, if $G$ and $H$ are connected
and for some $a,b\in[0,\infty]$ we have
\[
a\cdot d_G(x,y) \leq d_H(f(x),f(y))\leq b\cdot d_G(x,y)
\]
for all $x,y\in G$, then it is easy to see (\cf Remark following
Lemma~\ref{lem:complex-induced-map} below) that
\[
a'\cdot d_{\la(G)}(u,v) \leq d_{\la(H)}\big(\fb(u),\fb(v)\big) \leq
b'\cdot d_{\la(G)}(u,v)
\]
for all $u,v\in\la(G)$, where $a'=\min \{1,a\}$ and $b'=\max\{1,b\}$. 
Of course, this result is only interesting if $a>0$, $b<\infty$ or
both, \ie if $f$ is co-Lipschitz, Lipschitz or bi-Lipschitz,
respectively. In particular, if $(G,d_G)$ bi-Lipschitzly embeds into
$(H,d_H)$, then $\la(G)$ bi-Lipschitzly embeds into 
$\la(H)$. Observe that if $f$ is injective, then $b\geq 1$, and if in
addition $0<a\leq 1$, then $\dist(\fb)\leq\dist(f)$. However, there
are bi-Lipschitz embeddings of interest where $a\to\infty$ 
with $b/a$ bounded. In this case, $b'/a'$ gets arbitrarily large. For
this reason, we will consider more complicated induced maps in
Lemmas~\ref{lem:complex-induced-map}
and~\ref{lem:even-more-complex-induced-map} below.

\begin{lem}
  \label{lem:complex-induced-map} 
  Let $f\colon G\to H$ be a map between connected graphs $G$ and $H$,
  and let $a,b\in[0,\infty]$  be given so that 
  \[
  a\cdot d_G(x,y)\leq d_H(f(x),f(y))\leq b\cdot d_G(x,y)\qquad
  \text{for all }x,y\in G\ .
  \]
  Then for every $m\in\{0\}\cup\{1,\dots, \ceil{a/2}-1\}$, there is a
  map $\fb_m\colon \la(G)\to\la(H)$ induced by $f$ and $m$ such that
  $\fb_0=\fb$ and
  \[
  a'\cdot d_{\la(G)}(u,v)\leq d_{\la(H)}\big(\fb_m(u),\fb_m(v)\big)
  \leq b'\cdot d_{\la(G)}(u,v)
  \]
  for all $u,v\in\la(G)$, where $a'=\min\{a,m+1\}$ and
  $b'=\max\{b,3m+1\}$.
\end{lem}

\begin{rem} 
  Assume that $f\colon G\to H$ is a bi-Lipschitz embedding, and thus
  that $b\geq 1$ and $a>0$. The observations made before the statement
  of the lemma follow immediately by taking $m=0$. On the other hand,
  choosing $m=\ceil{a/2}-1$, and by considering the fractions
  $\frac{b}{a}$, $\frac{b}{m+1}$, $\frac{3m+1}{a}$ and 
  $\frac{3m+1}{m+1}$, it is easy to see that 
  $\dist(\fb_m)\leq b'/a'\leq 3b/a$. Hence we obtain the universal bound
  \[
  \frac{\dist\big( \fb_{\ceil{a/2}-1} \big)}{\dist(f)} \leq 3\ .
  \]
\end{rem}

\begin{proof}
  If $m=0$ we define $\fb_0$ to be the natural induced map $\fb$, and
  the conclusion holds with $a'=\min \{a,1\}$ and
  $b'=\max\{b,1\}$. Indeed, let $x_0,x_1,\dots,x_n$ be vertices of
  $G$, and let $y_i=f(x_i)$ for $0\leq i\leq n$. Consider a walk $w$
  of length $\ell$ in $G$ from $x_0$ to $x_n$ visiting
  $x_0,x_1,\dots,x_n$ in this order. For each $i=1,\dots,n$, let $w_i$
  be the section of $w$ from $x_{i-1}$ to $x_i$, and let $\ell_i$ be
  the length of $w_i$. Then $d_H(f(x_{i-1}),f(x_i))\leq b\cdot
  d_G(x_{i-1},x_i)\leq b\cdot\ell_i$ for each $i=1,\dots,n$. Hence
  there is a walk in $H$ of length at most $b\cdot\ell$ from $y_0$ to
  $y_n$ visiting $y_0,y_1,\dots,y_n$ in this order. An essentially
  identical argument shows that if there is a walk in $H$ of length
  $\ell$ from $y_0$ to $y_n$ visiting $y_0,y_1,\dots,y_n$ in this
  order, then there is a walk in $G$ of length at most $\ell/a$ from
  $x_0$ to $x_n$ visiting $x_0,x_1,\dots,x_n$ in this order. It
  follows that
  \begin{equation}
    \label{eq:tsp-in-H}
    a\cdot \tsp_G (x,C,y) \leq \tsp_H \big( f(x),f(C),f(y)\big) \leq
    b\cdot \tsp_G (x,C,y)
  \end{equation}
  for all $x,y\in G$ and for all finite $C\subset G$. Now fix vertices
  $(A,x)$ and $(B,y)$ of $\la(G)$. Observe that
  $f(A)\symdif f(B)\subset f(A\symdif B)$ and
  $\abs{f(A\symdif B)}\leq \abs{A\symdif B}$, and moreover equality
  holds when $a>0$. Hence, using
  Proposition~\ref{prop:graph-metric-in-lamplighter}
  and~\eqref{eq:tsp-in-H}, we have
  \begin{multline*}
    d_{\la(H)}\big( f(A,x),f(B,y) \big) =
    \tsp_H \big( f(x),f(A)\symdif f(B), f(y)\big) + \abs{f(A)\symdif
      f(B)}\\
    \leq b\cdot \tsp_G (x,A\symdif B,y)+\abs{A\symdif B} \leq b'\cdot
    d_{\la(G)} \big( (A,x),(B,y) \big)
  \end{multline*}
  and if $a>0$, then
  \begin{multline*}
    d_{\la(H)}\big( f(A,x),f(B,y) \big) =
    \tsp_H \big( f(x),f(A\symdif B), f(y)\big) + \abs{f(A\symdif B)}\\
    \geq a\cdot \tsp_G (x,A\symdif B,y)+\abs{A\symdif B} \geq a'\cdot
    d_{\la(G)} \big( (A,x),(B,y) \big)\ .
  \end{multline*}

  Assume now that $1\leq m\leq \ceil{a/2}-1$ and in particular that
  $a>0$. Set $a'=\min(a,m+1)$ and $b'=\max(b,3m+1)$. Without loss of
  generality we will assume that $G$ has at least two vertices. For
  every vertex $y$ in the image of $f$, choose a path
  $(u_0,u_1,u_2,\ldots, u_m)$ in $H$ starting at $u_0=y$, and set
  $W_y=\{ u_0, u_1,\dots,u_m\}$. This can always be done by picking
  another vertex $z\in f(G)$ and using $d_H(y,z)\geq a$ which in turn
  follows from the assumptions on $f$. It is easy to see that
  $\tsp_H(y,W_y,y)=2m$ since the unique optimal walk for the salesman
  is the path from $u_0$ to $u_m$ and back. Note also that the sets
  $W_y$, $y\in f(G)$, are pairwise disjoint, since the vertices in
  $f(G)$ are $a$-separated and $2m<a$. For a finite set
  $C\subset f(G)$ we put $W_C=\bigcup_{y\in C}  W_y$. Finally, we
  define $\fb_m\colon \la (G)\to \la(H)$ by letting
  $\fb_m(A,x)=\big(W_{f(A)}, f(x)\big)$.

  Given vertices $y,z$ and a finite subset $C$ in the image of $f$, we
  now obtain estimates on $\tsp_H(y,W_C,z)$. Since $C\subset W_C$, we
  immediately obtain  $\tsp_H (y,W_C,z) \geq \tsp_H (y,C,z)$. On the
  other hand, consider the following walk. Start with a walk $w$ in
  $H$ of length $\tsp_H(y,C,z)$ from $y$ to $z$ visiting all vertices
  of $C$, and each time $w$ visits a vertex $u\in C$, insert a walk of
  length $2m=\tsp_H(u,W_u,u)$ starting and ending at $u$ and visiting
  all vertices in $W_u$. The resulting walk from $y$ to $z$ visits all
  the vertices in $W_C$ and has length
  $\tsp_H(y,C,z)+2m\abs{C}$. Therefore, we have
  \begin{equation}
    \label{eq:tsp-in-H-inflated}
    \tsp_H (y,C,z)\leq\tsp_H (y,W_C,z) \leq \tsp_H(y,C,z)+2m\abs{C}\ .
  \end{equation}
  Let us now fix vertices $(A,x)$ and $(B,y)$ in $\la(G)$. Observe that
  $W_{f(A)}\symdif W_{f(B)}=W_{f(A)\symdif f(B)}$ and
  $f(A)\symdif f(B)=f(A\symdif B)$. Combining
  Proposition~\ref{prop:graph-metric-in-lamplighter},
  \eqref{eq:tsp-in-H-inflated} and~\eqref{eq:tsp-in-H}, we obtain
  \begin{align*}
    d_{\la(H)} \big( &\fb_m(A,x),\fb_m(B,y) \big) =
    \tsp_H \big( f(x),W_{f(A\symdif B)},f(y)\big) +
    \bigabs{W_{f(A \symdif B)}}\\
    &\leq \tsp_H \big(f(x),f(A\symdif B),f(y)\big) + 2m\cdot\abs{f(A\symdif B)}+(m+1)\cdot\abs{f(A\symdif B)}\\
    &\leq b\cdot \tsp_G (x,A\symdif B,y) + (3m+1)\cdot\abs{A\symdif B}\\
    &\leq b'\cdot d_{\la(G)} \big((A,x),(B,y)\big)\ ,
  \end{align*}
  and
  \begin{align*}
    d_{\la(H)} \big( \fb_m(A,x),\fb_m(B,y) \big) &\geq 
    \tsp_H \big( f(x),f(A\symdif B),f(y)\big) + \bigabs{W_{f(A \symdif B)}}\\
    &\geq a\cdot \tsp_G (x,A\symdif B,y) + (m+1)\cdot\abs{f(A\symdif B)}\\
    &\geq a'\cdot d_{\la(G)} \big((A,x),(B,y)\big)\ .
  \end{align*}
\end{proof}

In the last lemma of this section we consider a more sophisticated
construction in order to improve the bound on the distortion. This
construction is of a slightly different nature since it provides an
embedding with a higher degree of faithfulness at the expense that we
need to consider the lamplighter graph over a slightly bigger graph
that contains the original graph $H$ under scrutiny. In some specific
situations (see Proposition~\ref{prop:lamp-complete-into-binary}),
this turns out not to be an issue and
Lemma~\ref{lem:even-more-complex-induced-map} can be efficiently used
to significantly improve the distortion.

Let us fix a map $f\colon G\to H$ between two graphs. Let
$Q=(V,E,v_0)$ be a pointed graph and $W$ be a finite subset of $V$
with $v_0\in W$. Let $\Ht$ be the graph obtained by coalescing $H$
with $\abs{f(G)}$ copies of $Q$ as follows. For each vertex $y$ in the
image of $f$, we coalesce to $H$ at the vertex $y$ the copy of $Q$
that corresponds to $y$. This leads to a map
$\tilde{f}\colon\la(G)\to\la(\Ht)$ induced by $f$, $Q$ and 
$W$ and defined as follows. For $y\in f(G)$ we let $W_y$ denote
the set $W$ considered in the copy of $Q$ that corresponds to $y$, and
for a finite subset $C$ of $f(G)$ we let
$W_C=\bigcup_{y\in C} W_y$. Finally, for a vertex $(A,x)$ of $\la(G)$
define $\tilde{f}(A,x)=\big(W_{f(A)},f(x)\big)$.

\begin{lem}
  \label{lem:even-more-complex-induced-map}
  Let $f\colon G\to H$ be a map between connected graphs. Let
  $Q=(V,E,v_0)$ be a connected pointed graph and $W$ be a finite
  subset of $V$ with $v_0\in V$. Let $\Ht$ and
  $\tilde{f}\colon \la(G)\to\la(\Ht)$ be the map defined above. Assume
  that there exist $a,b\in[0,\infty]$ such that
  \[
  a\cdot d_G(x,y) \leq d_H\big(f(x),f(y)\big)\leq b\cdot d_G(x,y)
  \]
  for all $x,y\in G$. Then it follows for all $u,v\in\la(G)$ that
  \[
  a'\cdot d_{\la(G)}(u,v) \leq d_{\la(\Ht)}\big(\tilde{f}(u),\tilde{f}(v)\big) \leq
  b'\cdot d_{\la(G)}(u,v)\ ,
  \]
  where $a'=\min \{a,c\}$, $b'=\max\{b,c\}$ and
  $c=\tsp_Q(v_0,W,v_0)+\abs{W}$.
\end{lem}

\begin{rem}
  This result can be very versatile. Assume that $a,b\in\bn$ and
  $b-a\geq 2$. Assume further that $\abs{V}\geq b$. Then $W$ can be
  chosen so that $a\leq c\leq b$, and hence
  $\dist(\tilde{f})\leq b/a$. Indeed, given a finite $W\subset V$, replacing
  $W$ by $W\cup\{q\}$ for some $q\in V\setminus W$ that is joined to a
  vertex in $W$, the value of $c$ increases by at most~$3$. Hence,
  starting with $W=\{v_0\}$ and adding one vertex at a time, we
  eventually arrive at a set $W$ for which $a\leq c\leq b$ holds.
\end{rem}

\begin{proof}
  Given vertices $y,z$ and a finite subset $C$ in the image of $f$, an
  optimal solution for computing $\tsp_{\Ht}\big(y,W_C,z)$ can be
  obtained as follows. Start with a walk $w$ in $H$ of length
  $\tsp_H(y,C,z)$ from $y$ to $z$ visiting all vertices of $C$, and
  each time $w$ visits a vertex $u\in C$ insert a walk of length
  $\tsp_Q(v_0,W,v_0)$ in the copy of $Q$ corresponding to $u$ that
  starts and ends at $v_0$ and visits all vertices in $W$. The
  resulting walk is easily seen to be optimal, and hence yields the
  formula
  \begin{equation}
    \label{eq:tsp-in-Ht}
    \tsp_{\Ht}(y,W_C,z)=\tsp_H(y,C,z)+\abs{C}\cdot\tsp_Q(v_0,W,v_0)
    \ .
  \end{equation}
  Let us now fix vertices
  $(A,x)$ and $(B,y)$ in $\la(G)$. Observing that
  $W_{f(A)}\symdif W_{f(B)}=W_{f(A)\symdif f(B)}$, and combining
  Proposition~\ref{prop:graph-metric-in-lamplighter}
  with~\eqref{eq:tsp-in-Ht}, we obtain
  \begin{multline*}
    d_{\la(\Ht)} \big(\tilde{f}(A,x),\tilde{f}(B,y) \big) =
    \tsp_{\Ht}\big(f(x),W_{f(A)\symdif f(B)},f(y)\big) +
    \bigabs{W_{f(A)\symdif f(B)}}\\
    = \tsp_H\big(f(x),f(A)\symdif f(B),f(y)\big) +
    \big(\tsp_Q(v_0,W,v_0)+\abs{W}\big)\cdot\abs{f(A)\symdif f(B)}\ .
  \end{multline*}
  As before, we have $f(A)\symdif f(B)\subset f(A\symdif B)$ and
  $\abs{f(A\symdif B)}\leq \abs{A\symdif B}$, and moreover equality
  holds when $a>0$. Hence, by~\eqref{eq:tsp-in-H}, which is still
  valid in this context, we obtain
  \begin{align*}
    d_{\la(\Ht)} \big(\tilde{f}(A,x),\tilde{f}(B,y) \big) &\leq b\cdot \tsp_G
    (x,A\symdif B,y) + c\cdot\abs{A\symdif B}\\
    &\leq b'\cdot d_{\la(G)} \big((A,x),(B,y)\big)\ ,
  \end{align*}
  and if $a>0$, then
  \begin{align*}
    d_{\la(\Ht)} \big(\tilde{f}(A,x),\tilde{f}(B,y) \big) &\geq a\cdot \tsp_G
    (x,A\symdif B,y) + c\cdot\abs{f(A\symdif B)}\\
    &\geq a'\cdot d_{\la(G)} \big((A,x),(B,y)\big)\ .
  \end{align*}
\end{proof}

\section{Binary trees and Hamming cubes in lamplighter graphs}
\label{sec:5}

It was observed in~\cite{lpp:96} that the lamplighter group
$\bz_2\wr\bz$ contains a bi-Lipschitz copy of the infinite binary
tree. We provide a simple proof of the finite version of this fact.

\begin{lem}
  \label{lem:tree-into-lamplighter-of-path}
  Let $k\in\bn$. Then $\bin_k$ bi-Lipschitzly embeds with distortion
  at most $2$ into $\la(\pa_k)$.
\end{lem}

\begin{proof}
  Let $v_0,\dots,v_k$ be the vertices of $\pa_k$ with edges
  $v_{i-1}v_i$ for $1\leq i\leq k$. For any
  $\vare=(\vare_1,\dots,\vare_n)\in \bin_k$, let 
  $A_{\vare}=\{v_{s-1}:\,\vare_s=1\}$, and define
  $f\colon \bin_k\to\la(\pa_k)$ by setting
  $f(\vare)=(A_\vare,v_{\abs{\vare}})$. We show that $f$ is a
  bi-Lipschitz embedding with distortion at most~2.

  Let us fix $\delta,\vare\in \bin_k$ and assume without loss of
  generality that $\abs{\delta}\leq \abs{\vare}$. Then by
  Proposition~\ref{prop:graph-metric-in-lamplighter} we have
  \begin{equation}
    d_{\la(\pa_k)}\big(f(\delta),f(\vare)\big) =
    \tsp_{\pa_k}\big(v_{\abs{\delta}},A_\delta\symdif
    A_\vare,v_{\abs{\vare}}) + \abs{A_\delta\symdif A_\vare}\ .
  \end{equation}
  Assume that $\delta$ and $\vare$ are adjacent, and thus
  $\vare=(\delta,\delta_{m+1})$, where $m=\abs{\delta}$. If
  $\delta_{m+1}=0$, then $A_\delta\symdif A_\vare=\emptyset$,
  otherwise $A_\delta\symdif A_\vare=\{v_m\}$. Therefore,
  \[
  d_{\la(\pa_k)}\big(f(\delta),f(\vare)\big)=
  \tsp_{\pa_k}\big(v_m,A_\delta\symdif A_\vare,v_{m+1}\big) +
  \abs{A_\delta\symdif A_\vare}\leq 1+1=2\ ,
  \]
  and thus $f$ is $2$-Lipschitz.

  We now derive the lower bound. Let $\delta\meet\vare$ denote the
  last common ancestor of $\delta$ and $\vare$. Thus,
  $\delta\meet\vare=(\delta_1,\dots,\delta_r)$, where
  $r=\max\{ i\geq 0:\, \delta_j=\vare_j \text{ for }1\leq j\leq i\}$.
  Since by definition we have
  $A_\delta\symdif
  A_\vare\subset\{v_{\abs{\delta\meet\vare}},\dots,v_{\abs{\vare}-1}\}$,
  an optimal solution for computing
  $\tsp_{\pa_k}(v_{\abs{\delta}},A_\delta\symdif
  A_\vare,v_{\abs{\vare}})$ starts at $v_{\abs{\delta}}$, then travels
  to $v_{\abs{\delta\meet\vare}}$, and finally travels to
  $v_{\abs{\vare}}$. Thus,
  \begin{align*}
    \tsp_{\pa_k}
    \big(v_{\abs{\delta}},A_\delta\symdif A_\vare,v_{\abs{\vare}}\big)
    &=
    d_{\pa_k}\big(v_{\abs{\delta}},v_{\abs{\delta\meet\vare}}\big)
    +d_{\pa_k}\big(v_{\abs{\delta\meet\vare}},v_{\abs{\vare}}\big)\\
    &=(\abs{\delta}-\abs{\delta\meet\vare})
    +(\abs{\vare}-\abs{\delta\meet\vare})\\
    &=d_{\bin}(\delta,\vare).
  \end{align*}
  It follows that 
  \begin{equation}
    d_{\la(\pa_k)}\big(f(\delta),f(\vare)\big) \geq
    d_{\bin}(\delta,\vare)\ .
  \end{equation}
\end{proof}
It is clear that a similar argument as in the proof above shows that
the lamplighter graph $\la(\pa_\infty)$ over the infinite path
$\pa_\infty$ contains a bi-Lipschitz copy of the infinite binary
tree. Since $\la(\pa_\infty)$ and $\bz_2\wr \bz$ are isometric (with
suitable choice of generators for $\bz_2\wr\bz$), the observation
from~\cite{lpp:96} can be
recovered. Lemma~\ref{lem:tree-into-lamplighter-of-path} also provides
the final result we need to complete the proof of
Theorem~\ref{mainthm:B}.

\begin{proof}[Proof of Theorem~\ref{mainthm:B}]
  The implications ``(ii)$\implies$(i)'' and ``(iii)$\implies$(i)''
  follow from Propositions~\ref{prop:star-characterization}
  and~\ref{prop:rose-characterization}, respectively. To establish the
  reverse implications, fix
  $n\in\bn$. Observe that for each $k\in\bn$ the graphs $\st_{n,k}$
  and $\rose_{n,2k}$ contain isometric copies of $\pa_k$, and hence by
  combining Lemmas~\ref{lem:complex-induced-map}
  and~\ref{lem:tree-into-lamplighter-of-path}, the binary tree
  $\bin_k$ bi-Lipschitzly embeds with distortion at most~2 into the
  lamplighter graphs $\la(\st_{n,k})$ and $\la(\rose_{n,2k})$. The
  implications ``(i)$\implies$(ii)'' and ``(i)$\implies$(iii)'' now
  follow from Bourgain's metric characterization of
  superreflexivity~\cite{bourgain:86}.
\end{proof}

We now turn to the embeddability of Hamming cubes into lamplighter
graphs. Here $\com_n$, for $n\in\bn$, denotes the complete graph on $n$
vertices.

\begin{lem}
  \label{lem:hamming-into-lamplighter-of-complete}
  Let $k,m\in\bn$. Then $\ham_k$ bi-Lipschitzly embeds into
  $\la(\com_{km})$ with distortion at most $1+\frac1{2m}$.  
\end{lem}

\begin{proof}
  Recall that $\ham_k$ can be thought of as the set of all subsets of
  $\{1,\dots,k\}$ and that under this identification the Hamming
  metric becomes the symmetric difference metric. Let us now partition
  the vertex set of $\com_{km}$ into $k$ sets $V_1,\dots,V_k$ each of
  size $m$, and let us also fix a vertex $v_0$ of $\com_{km}$. Define
  $f\colon \ham_k\to\la(\com_{km})$ by setting $f(I)=(V_I,v_0)$, where
  $V_I=\bigcup_{i\in I} V_i$.

  To estimate the distortion of $f$, let us fix distinct elements
  $I,J\in\ham_k$. Note that $V_I\symdif V_J=V_{I\symdif J}$, and hence
  \[
  \abs{V_I\symdif V_J}=m\abs{I\symdif J}=m\cdot d_{\ham}(I,J)\ .
  \]
  It follows that
  \[
  \tsp_{\com_{km}} (v_0,V_I\symdif V_J,v_0)=%
  \begin{cases}
    m d_{\ham}(I,J) & \text{if }v_0\in V_I\symdif V_J,\\
    m d_{\ham}(I,J)+1 & \text{if }v_0\notin V_I\symdif V_J\ .
  \end{cases}
  \]
  Combining the above with
  Proposition~\ref{prop:graph-metric-in-lamplighter} yields
  \[
  2m\cdot d_{\ham}(I,J) \leq d_{\la(\com_{km})} \big(f(I),f(J)\big)
  \leq 2m\cdot d_{\ham}(I,J) +1 \leq (2m+1)\cdot d_{\ham}(I,J)\ .
  \]
\end{proof}

\begin{rem}
  Lemma~\ref{lem:hamming-into-lamplighter-of-complete} shows in
  particular that for every $k\in\bn$, there is a bi-Lipschitz
  embedding of $\ham_k$ into $\la(\com_k)$ of distortion at
  most~$\frac32$, and that for every $k\in\bn$ and $\vare>0$, there
  exists $n\in\bn$ such that $\ham_k$ bi-Lipschitzly embeds into
  $\la(\com_n)$ with distortion at most~$1+\vare$, and moreover $n$
  can be chosen to be $\frac{k}{2\vare}$.
\end{rem}

At this point, we need one more ingredient to prove
Theorem~\ref{mainthm:C}, which is the following well known fact.

\begin{lem}
  \label{lem:complete-into-binary}
  Let $k\in\bn$ and $\vare>0$. Then $\com_k$ embeds with distortion at
  most $1+\vare$ into $\bin_n$ whenever
  $n\geq \log_2(k)\cdot\frac{1+\vare}{\vare}$.
\end{lem}

\begin{proof}
  Choose $s,t\in\bn$ such that $2^s\geq k$ and
  $\frac{s+t}{t+1}<1+\vare$. We show that $n=s+t$ works.

  By a \emph{leaf }of the binary tree $\bin_s$ of height $s$, we mean
  a vertex $\vare$ with $\abs{\vare}=s$.  The binary tree $\bin_n$ of
  height $n=s+t$ can be considered as being constructed by coalescing
  $2^s$ copies of the binary tree $\bin_t$ to the leaves of the binary
  tree $\bin_s$ as follows. For each leaf $\vare$ of $\bin_s$, we
  coalesce a copy of $\bin_t$ at $\emptyset$, its root, to $\bin_s$ at
  $\vare$.

  Pick $k$ leaves $\ell_1,\dots,\ell_k$ of $\bin_n$, one from each of
  $k$ different copies of $\bin_t$. Let $v_1,\dots,v_k$ be the
  vertices of $\com_k$, and define $f\colon\com_k\to\bin_n$ by
  $f(v_i)=\ell_i$, $i=1,\dots,k$. We then have
  \[
  2t+2\leq d_{\bin} \big(f(v_i),f(v_j)\big)\leq
  \diam(\bin_{s+t})=2(s+t)
  \]
  for all $i\neq j$. Thus, $f$ has distortion at most
  $\frac{s+t}{t+1}$, which in turn is at most $1+\vare$ by the choice
  of $s$ and $t$.
\end{proof}

\begin{proof}[Proof of Theorem~\ref{mainthm:C}]
  It follows from Theorem~\ref{mainthm:A} that $\la(\bin_k)$ embeds
  with distortion at most~$6$ into a finite Hamming cube. In turn, by
  Lemma~\ref{lem:hamming-into-lamplighter-of-complete}, the Hamming
  cube $\ham_k$ embeds into $\la(\com_k)$ with distortion at
  most~$\frac32$. It remains to show that $\big(\la(\com_k)\big)_{k\in\bn}$
  equi-bi-Lipschitzly embeds into $\big(\la(\bin_k)\big)_{k\in\bn}$,
  but this is true due to Lemma~\ref{lem:complete-into-binary}
  combined with Lemma \ref{lem:complex-induced-map}.
\end{proof}

The equi-bi-Lipschitz embeddability of
$\big(\la(\com_k)\big)_{k\in\bn}$ into
$\big(\la(\bin_k)\big)_{k\in\bn}$ can be made quantitatively more
precise using Lemma~\ref{lem:even-more-complex-induced-map}.

\begin{prop}
  \label{prop:lamp-complete-into-binary}
  Let $k\in\bn$ and $\vare>0$. Then $\la(\com_k)$ embeds with
  distortion at most $1+\vare$ into $\la(\bin_{N})$ whenever
  $N>n+\log_2 n+1$ and $n\geq \log_2(k)\cdot\frac{1+\vare}{\vare}$.
\end{prop}

\begin{proof}
  For $k=1,2$, it is clear that $\com_k$ embeds isometrically into
  $\bin_k$, and hence the same holds for the corresponding lamplighter
  graphs. We 
  now assume that $k\geq 3$ and follow the notation from the proof of
  Lemma~\ref{lem:complete-into-binary}. We have $s\geq 2$, and hence
  $2(s+t)-(2t+2)=2s-2\geq 2$. Choose $r\in\bn$ with $2^r>2(s+t)$, and
  let $Q$ be the pointed graph $(\bin_r,\emptyset)$. It follows from
  Lemma~\ref{lem:even-more-complex-induced-map} and the subsequent remark that
  there is a subset $W$ of the vertices of $Q$ such that the map
  $\tilde{f}\colon \la(\com_k)\to\la\big(\widetilde{\bin}_n\big)$, induced by
  $f$, $Q$ and $W$, has distortion at most
  $\frac{s+t}{t+1}<1+\vare$. Finally observe that, since the image of
  $f$ is contained in the set of leaves of $\bin_n$, it follows that
  $\widetilde{\bin}_n$ isometrically embeds into $\bin_{n+r}$, which
  in turn implies that $\la\big(\widetilde{\bin}_n\big)$ isometrically
  embeds into $\la(\bin_{n+r})$.
\end{proof}

\section{Conclusions}

Assume that a sequence $(G_k)_{k\in\bn}$ of graphs equi-bi-Lipschitzly contains
$(K_k)_{k\in\bn}$. It follows then from Theorem~\ref{mainthm:C} and
Lemma~\ref{lem:complex-induced-map}, together with the remark
thereafter, that the sequence $(\ham_k)_{k\in\bn}$ of Hamming cubes
equi-bi-Lipschitzly embeds into $(\la(G_k))_{k\in\bn}$. We do not know
if the converse holds.
 
\begin{problem}
  \label{problem:7.1}
  Given a sequence $(G_k)_{k\in\bn}$ of graphs, if the Hamming cubes
  $(\ham_k)_{k\in\bn}$ equi-bi-Lipschitzly embed into
  $(\la(G_k))_{k\in\bn}$, does it follow that $(\com_k)_{k\in\bn}$
  equi-bi-Lipschitzly embeds into $(G_k)_{k\in\bn}$?
\end{problem} 

The following tree might be a counterexample to
Problem~\ref{problem:7.1}.

\begin{ex}
  We construct a tree which can be seen as a ``binary tree with
  variable-size legs'' as follows. Given $k\in\bn$ and
  $\thickbar{\ell}=(\ell_1,\ell_2,\dots, \ell_k)\in\bn^k$, replace
  each edge on the $j^{\text{th}}$ level of the binary tree of length
  $k$ by a path of length $\ell_j$, where by an edge on the
  $j^{\text{th}}$ level we mean an edge such that the distance from
  its farthest endpoint to the root is $j$. Denote by
  $\bin_{\thickbar{\ell}}$ the new tree, of length
  $\ell=\sum_{i=1}^k{\ell_i}$, thus obtained. The tree
  $\bin_{\thickbar{\ell}}$, with $\thickbar{\ell}=(4,2,1)$ is the tree
  of length $\ell=7$ depicted in the illustration below.

  \begin{figure}[ht]
    \label{fig:sstree}
    \caption{$\mathrm{B}_{\bar{\ell}}$, with $\bar{\ell}=(4,2,1)$}
    \vskip .5cm
    \includegraphics[scale=.3]{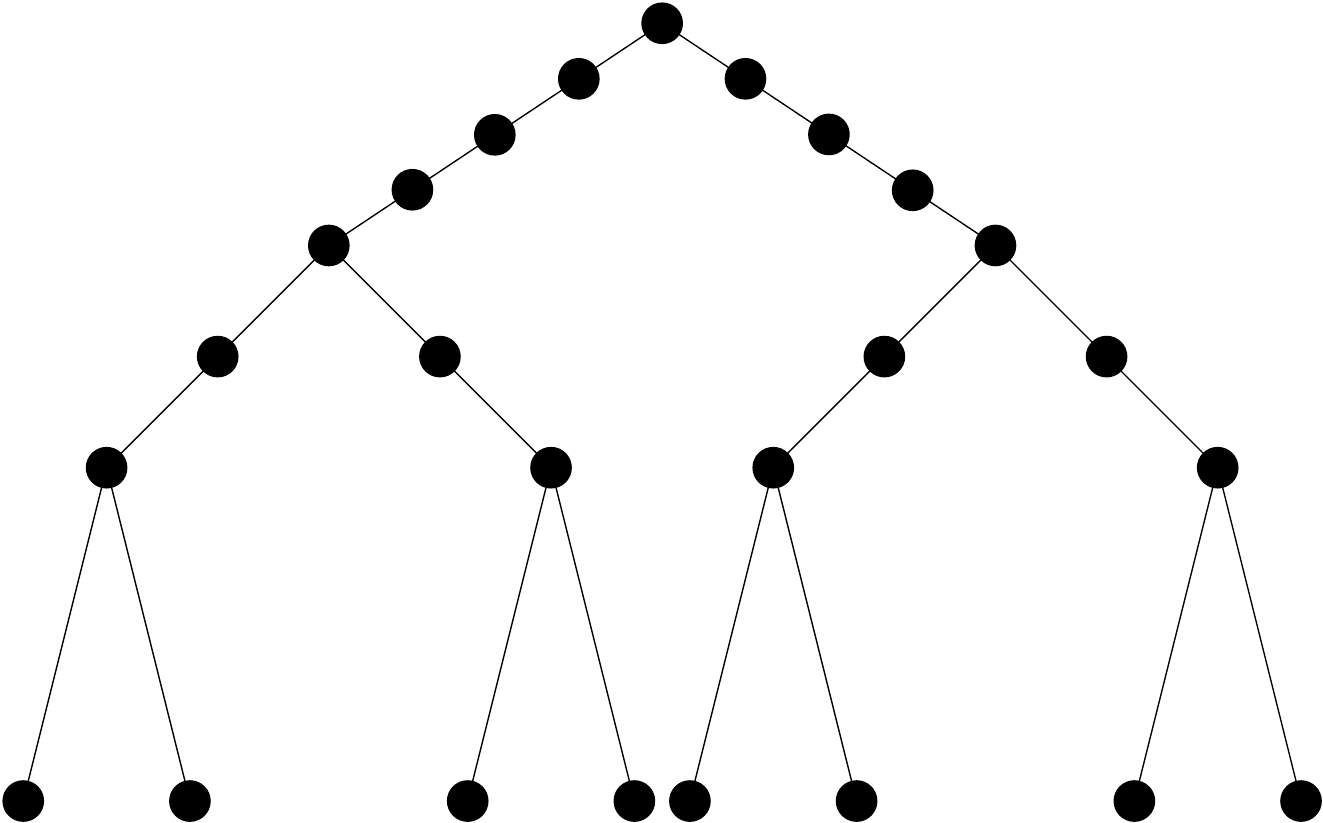}
  \end{figure}
\end{ex}

If we choose for every $k\in\bn$, the sequence
$\thickbar{\ell}^{(k)}=\big(\ell_1^{(k)}, \ell_2^{(k)},\dots,
\ell^{(k)}_k\big)$ so that $\ell_1^{(k)}$  is chosen large enough
compared to $\ell_2^{(k)}$, $\ell_2^{(k)}$ is chosen large enough
compared to $\ell_3^{(k)}$, \etc it is not hard, but cumbersome, to
prove that the sequence
$\big(\bin_{\thickbar{\ell}^{(k)}}\big)_{k\in\bn}$ does not
equi-bi-Lipschitzly contain $(\com_k)_{k\in\bn}$. So for this example
to become a counterexample to Problem~\ref{problem:7.1}, we need a positive
answer to the following question.

\begin{problem}
  \label{problem:7.3}
  Let $\big(\bin_{\thickbar{\ell}^{(k)}}\big)_{k\in\bn}$ constructed
  as in the description above. Does $(\ham_k)_{k\in\bn}$
  equi-bi-Lipschitzly embed into
  $\big(\la(\bin_{\thickbar{\ell}^{(k)}})\big)_{k\in\bn}$?
\end{problem} 

Any counterexample to Problem~\ref{problem:7.1} would be a
counterexample to the following problem.

\begin{problem}
  \label{problem:7.4} 
  If $(G_k)_{k\in\bn}$ is a sequence of graphs which does not
  equi-bi-Lipschitzly contain $(\com_k)_{k\in\bn}$, and if $X$ is a
  non-reflexive Banach space, does it follow that the sequence
  $\big(\la(G_k)\big)_{k\in\bn}$ equi-bi-Lipschitzly embed into $X$?
\end{problem}

Indeed, there are non-reflexive Banach spaces $X$ with non-trivial
type (\cf~\cite{james:74},~\cite{james:78} or~\cite{pisier-xu:87}). By
the observation at the beginning of this section, these spaces cannot
equi-bi-Lipschitzly contain sequences of graphs which
equi-bi-Lipschitzly contain $(\ham_k)_{k\in\bn}$.

\end{document}